\documentclass[twoside, a4paper, 12pt]{amsart}
\usepackage[foot]{amsaddr}
\usepackage{fullpage}
\usepackage{amsfonts}
\usepackage{amssymb}
\usepackage{enumitem}
\theoremstyle{plain}
\newtheorem*{claim*}{Claim}
\newtheorem{thm}{Theorem}[section]
\newtheorem{corollary}[thm]{Corollary}
\newtheorem{lemma}[thm]{Lemma}
\newtheorem{prop}[thm]{Proposition}
\theoremstyle{definition}
\newtheorem{defn}[thm]{Definition}
\newtheorem{ex}[thm]{Example}

\newtheorem{remark}[thm]{Remark}

\newtheorem{prob}[thm]{Open Problem}

\begin{document}
\title{\large{Separability conditions in acts over monoids}}
\author{Craig Miller}
\address{Department of Mathematics, University of York, UK, YO10 5DD}
\email{craig.miller@york.ac.uk}
\maketitle

\begin{abstract}
We discuss residual finiteness and several related separability conditions for the class of monoid acts, namely weak subact separability, strong subact separability and complete separability.  For each of these four separability conditions, we investigate which monoids have the property that all their (finitely generated) acts satisfy the condition.  In particular, we prove that: all acts over a finite monoid are completely separable (and hence satisfy the other three separability conditions); all finitely generated acts over a finitely generated commutative monoid are residually finite and strongly subact separable (and hence weakly subact separable); all acts over a commutative idempotent monoid are residually finite and strongly subact separable; and all acts over a Clifford monoid are strongly subact separable.
\end{abstract}
\vspace{0.5em}

\section{\large{Introduction}\nopunct}
\label{sec:intro}

The notion of separability concerns separating an element from a subset in a finite homomorphic image.  The most well-known concept regarding separability is residual finiteness.  An algebra $A$ is said to be {\em residually finite} if for any pair of distinct elements $a, b\in A$ there exist a finite algebra $C$ and a homomorphism $\theta : A\to C$ such that $a\theta\neq b\theta.$  Residual finiteness has been studied extensively within many areas of algebra and has proved to be a powerful tool.  For example, Evans proved that any finitely presented, residually finite algebra has a solvable word problem \cite{Evans}.\par
Residual finiteness is an instance of what we call a {\em separability condition}.  Separability conditions arise from the idea of separating elements from subsets of a particular type (e.g.\ singletons, subalgebras, etc.).
More precisely, let $\mathcal{K}$ be a class of algebras, let $A\in\mathcal{K},$ and let $\mathcal{S}$ be a collection of non-empty subsets of $A.$  We say that $A$ satisfies the {\em separability condition with respect to $\mathcal{S}$} if for any $X\in\mathcal{S}$ and any $a\in A\!\setminus\!X,$ there exist a finite algebra $C\in\mathcal{K}$ and a homomorphism $\theta : A\to C$ such that $a\theta\notin X\theta.$
In this case we say that $a$ can be {\em separated} from $X$ and that $\theta$ {\em separates} $a$ from $X.$  Of course, this is equivalent to there existing a finite index congruence $\rho$ on $A$ such that $(a, x)\notin\rho$ for all $x\in X.$  (The {\em index} of an equivalence relation is the number of classes.)  Likewise, we say that $\rho$ {\em separates} $a$ from $X.$\par
Clearly residual finiteness may be viewed as the separability condition with respect to the collection of all singleton subsets.  If the class $\mathcal{K}$ is a variety, then being residually finite is equivalent to satisfying the separability condition with respect to the collection of all {\em finite} subsets.  
Along with residual finiteness, we shall consider several other separability conditions.  These conditions have been studied in a variety of areas of algebra and under many different names.  The following names, introduced in \cite{Miller}, are designed to describe the conditions and highlight the relationships between them.
An algebra $A$ is said to be:
\begin{itemize}[leftmargin=*]
\item {\em weakly subalgebra separable} if it satisfies the separability condition with respect to the collection of all finitely generated subalgebras;
\item {\em strongly subalgebra separable} if it satisfies the separability condition with respect to the collection of all subalgebras;
\item {\em completely separable} if it satisfies the separability condition with respect to the collection of all non-empty subsets.
\end{itemize}
When we are working with a particular class of algebraic structures, we replace the word `subalgebra' in `weakly/strongly subalgebra separable' with the specific type of subalgebra.  For instance, for the class of groups we use the term weakly {\em subgroup} separable.\par
In group theory, the above separability conditions have received considerable attention (except complete separability, which is equivalent to being finite for groups).  We refer the reader to \cite{Magnus} for a survey of residually finite groups.  For weak subgroup separability and strong subgroup separability, one may consult \cite{Robinson} and the references therein.
In semigroup theory, residual finiteness has been studied in, for instance, \cite{Carlisle, Golubov1, Golubov2, Golubov3, Gray1, Gray2, Lallement, Lesohin}, weak subsemigroup separability in \cite{Miller, Gerry}, strong subsemigroup separability in \cite{Golubov1, Golubov2, Golubov3, Lesohin, Miller, Gerry}, and complete separability in \cite{Golubov1, Miller, Gerry}.
All of the aformentioned separability conditions have also been considered for monounary algebras in \cite{Witt}.
Residual finiteness has been studied for many other algebraic structures, including rings \cite{Faith}, modules \cite{Vara}, Lie algebras \cite{Bahturin, Premet, Zaicev} and lattices \cite{Adams}.  Most pertinent for this article, Kozhukhov investigated the property that all acts over a semigroup are residually finite in \cite{Kozhukhov1, Kozhukhov2, Kozhukhov3}.\par
The purpose of this paper is to study, for the class of monoid acts, the properties of residual finiteness, weak subact separability, strong subact separability and complete separability, and to investigate certain related finiteness conditions for the class of monoids.  The paper is structured as follows.  In Section \ref{sec:prelim}, we establish some basic definitions and facts about monoid acts and the separability conditions.  In Section \ref{sec:cyclic}, we first consider the separability conditions for cyclic acts, and we then discuss how these conditions relate to one another and whether they are preserved under disjoint unions.  The purpose of Section \ref{sec:EF} is to present equivalent formulations of each of the four separability conditions.  In Sections \ref{sec:FC} and \ref{sec:regular} we study certain natural finiteness conditions relating to the separability conditions.  Specifically, given any of the four separability conditions $\mathcal{C},$ we investigate which monoids have the property that all their acts satisfy $\mathcal{C},$ and which monoids have the property that all their {\em finitely generated} acts satisfy $\mathcal{C}.$  We find some equivalent characterisations of these finiteness conditions in Section \ref{sec:FC}.  Section \ref{sec:regular} is concerned with certain special classes of monoids, namely commutative monoids, groups, completely simple semigroups with identity adjoined, and Clifford monoids.

\section{\large{Preliminaries}\nopunct}
\label{sec:prelim}

\subsection{Monoid acts\nopunct}
~\par\vspace{0.5em}

Let $M$ be a monoid with identity 1.  A {\em right $M$-act} is a non-empty set $A$ together with a map $$A\times M\to A, (a, m)\mapsto am$$
such that $a(mn)=(am)n$ and $a1=a$ for all $a\in A$ and $m, n\in M.$
Left $M$-acts are defined analogously.  However, we will deal only with {\em right} $M$-acts and will therefore just speak of {\em $M$-acts.}\par
Given a non-empty set $A,$ we denote by $T_A$ the {\em full transformation monoid} on $A$; that is, the set of all transformations of $A$ under composition of mappings.  $M$-acts are in one-to-one correspondence with representations of the monoid $M$ by transformations of sets:

\begin{prop}[{\cite[Proof of Proposition 1.4.4]{Kilp}}]
\label{rep}
Let $M$ be a monoid.  Given an $M$-act $A,$ we obtain a monoid homomorphism $\theta : M\to T_A$ by defining $a(m\theta)=am$ for all $a\in A$ and $m\in M.$  Conversely, given a monoid homomorphism $\theta : M\to T_A,$ the set $A$ is turned into an $M$-act by defining $am=a(m\theta)$ for all $a\in A$ and $m\in M.$ 
\end{prop}

Notice that for any monoid $M,$ the set $M$ itself is an $M$-act via right multiplication.  For clarity, we denote this $M$-act by $\textbf{M}.$\par
A non-empty subset $B$ of an $M$-act $A$ is a {\em subact} of $A$ if $bm\in B$ for all $b\in B$ and $m\in M.$
Note that the right ideals of $M$ are precisely the subacts of $\textbf{M}.$\par
An element $a$ of an $M$-act $A$ is called a {\em zero} if $am=a$ for all $m\in M.$
Thus, if $a\in A$ is a zero, the singleton $\{a\}$ is a subact of $A.$
We note that monoid acts can have more than one zero.\par
For two $M$-acts $A$ and $B,$ a map $\theta : A\to B$ is an {\em $M$-homomorphism} if $(am)\theta=(a\theta)m$ for all $a\in A$ and $m\in M$. 
If $\theta$ is also bijective, then it is an {\em $M$-isomorphism}, and we write $A\cong B.$\par
An equivalence relation $\rho$ on $A$ is an ({\em $M$-act}) {\em congruence} on $A$ if $(a, b)\in\rho$ implies $(am, bm)\in\rho$ for all $a, b\in A$ and $m\in M.$
For a congruence $\rho$ on an $M$-act $A,$ the quotient set $$A/\rho=\{[a]_{\rho} : a\in A\}$$
becomes an $M$-act by defining $[a]_{\rho}m=[am]_{\rho}$ for all $[a]_{\rho}\in A/\rho$ and $m\in M.$
We call $A/\rho$ the {\em quotient} of $A$ by $\rho.$\par
Given an $M$-homomorphism $\theta : A\to B,$ we define a relation $\text{ker }\theta$ on $A$ by $$(a, b)\in\text{ker }\theta\iff a\theta=b\theta.$$
The relation $\text{ker }\theta$ is called the {\em kernel} of $\theta$ and is a congruence on $A.$\par
We have the following First Isomorphism Theorem for acts.
\begin{thm}[{\cite[Theorem 1.4.21]{Kilp}}]
Let $M$ be a monoid, let $A$ and $B$ be two $M$-acts, and let $\theta : A\to B$ be an $M$-homomorphism.  Then $A/\emph{ker }\theta\cong\emph{Im }\theta.$
\end{thm}
Given an $M$-act $A$ and a subact $B$ of $A,$ we define the {\em Rees congruence} $\rho_B$ on $A$ by 
$$(a, b)\in\rho_B\iff a=b\text{ or }a, b\in B.$$
We denote the quotient act $A/{\rho_B}$ by $A/B$ and call it the {\em Rees quotient} of $A$ by $B.$
We will usually identify the $\rho_B$-class $\{a\}\in A/B$ with $a$ for each $a\in A\!\setminus\!B,$ and denote the $\rho_B$-class $B\in A/B$ by $0_B$.\par
A subset $U$ of an $M$-act $A$ is said to be a {\em generating set} for $A$ if for any $a\in A$ there exist $u\in U$ and $m\in M$ such that $a=um.$
We write $A=\langle U\rangle$ if $U$ is a generating set for $A.$
An $M$-act $A$ is said to be {\em finitely generated} (resp.\ {\em cyclic}) if it has a finite (resp.\ one element) generating set.
There is a one-to-one correspondence between the set of cyclic $M$-acts and the set of right congruences on $M.$

\begin{prop}[{\cite[Proposition 1.5.17]{Kilp}}]
\label{cyclic act, right congruence}
Let $M$ be a monoid.  An $M$-act $A$ is cyclic if and only if there exists a right congruence $\rho$ on $M$ such that $A\cong\textbf{M}/\rho.$
\end{prop}

Given an $M$-act $A,$ there is a natural preorder $\leq_A$ on $A$ given by $$a\leq_A b\iff\langle a\rangle\subseteq\langle b\rangle.$$
The preorder $\leq_A$ induces an equivalence relation $\mathcal{R}_A$ on $A$ given by 
$$(a, b)\in\mathcal{R}_A\iff a\leq_A b\text{ and }b\leq_A a\iff\langle a\rangle=\langle b\rangle.$$
We shall refer to this equivalence relation as {\em Green's relation} $\mathcal{R}_A$ on $A.$  We use the term `Green's relation' since the relation $\mathcal{R}_{\textbf{M}}$ on $\textbf{M}$ coincides with Green's relation $\mathcal{R}$ on $M.$

\subsection{Separability conditions\nopunct}
~\par\vspace{0.5em}

We now provide a couple of basic facts about separability in general algebraic structures.

\begin{prop}
\label{algebra separability}
For an algebra $A$ the following statements hold.
\begin{enumerate}[leftmargin=*]
\item If $A$ is completely separable then it is strongly subalgebra separable.
\item If $A$ is strongly subalgebra separable then it is weakly subalgebra separable.
\item If $A$ is completely separable then it is residually finite.
\end{enumerate}
\end{prop}

\begin{proof}
The statements (1), (2) and (3) follow immediately from the definitions.
\end{proof}

\begin{lemma}
\label{subalgebra}
Let $A$ be an algebra, let $B$ be a subalgebra of $A,$ and let $\mathcal{C}$ be any of the following separability conditions: residual finiteness, weak subalgbra separability, strong subalgebra separability, complete separability.  If $A$ satisfies $\mathcal{C}$ then so does $B.$
\end{lemma}

\begin{proof}
Let $X\subseteq B$ be of the type associated with $\mathcal{C}$ and let $a\in B\!\setminus\!X.$  Since $A$ satisfies $\mathcal{C},$ there exists a finite index congruence $\rho$ on $A$ that separates $a$ from $X.$  Then $\rho\cap(B\times B),$ the restriction of $\rho$ to $B,$ is a finite index congruence on $B$ that separates $a$ from $X.$  Thus $B$ satisfies $\mathcal{C}.$
\end{proof}

Focussing now on monoid acts, for clarity we restate the aforementioned separability conditions in the setting of acts.

\begin{itemize}[leftmargin=*]
\item An $M$-act $A$ is said to be {\em residually finite} if for any pair of distinct elements $a, b\in A,$ there exist a finite $M$-act $C$ and an $M$-homomorphism $\theta : A\to C$ such that $a\theta\neq b\theta.$
\item An $M$-act $A$ is {\em weakly subact separable} if for any finitely generated subact $B$ of $A$ and any $a\in A\!\setminus\!B,$ there exist a finite $M$-act $C$ and an $M$-homomorphism $\theta : A\to C$ such that $a\theta\notin B\theta.$
\item An $M$-act $A$ is {\em strongly subact separable} if for any subact $B$ of $A$ and any $a\in A\!\setminus\!B,$ there exist a finite $M$-act $C$ and an $M$-homomorphism $\theta : A\to C$ such that $a\theta\notin B\theta.$
\item An $M$-act $A$ is {\em completely separable} if for any non-empty subset $X\subseteq A$ and any $a\in A\!\setminus\!X,$ there exist a finite $M$-act $C$ and an $M$-homomorphism $\theta : A\to C$ such that $a\theta\notin X\theta.$
\end{itemize}

In the remainder of the paper, when we speak of the ``separability conditions" we are referring to the above properties.

\begin{remark}
Note that an $M$-act $A$ is completely separable if and only if each $a\in A$ can be separated from $A\!\setminus\!\{a\}.$
\end{remark}

\section{\large{Cyclic Acts and Disjoint Unions}\nopunct}
\label{sec:cyclic}

If an algebra is weakly subalgebra separable, then it is clearly {\em cyclic subalgebra separable}, i.e.\ it satisfies the separability condition with respect to the collection of all cyclic (that is, 1-generated) subalgebras.  For various classes of algebras, including groups and semigroups, cyclic subalgebra separability is strictly weaker than weak subalgebra separability.  For monoid acts, however, these notions coincide.

\begin{prop}
\label{WSS=CSS}
Let $M$ be a monoid and let $A$ be an $M$-act.  Then $A$ is weakly subact separable if and only if it is cyclic subact separable.
\end{prop}

\begin{proof}
The direct implication is obvious.  For the converse, let $B$ be a subact of $A$ with a finite generating set $X,$ and let $a\in A\!\setminus\!B.$  Since $A$ is cyclic subact separable, for each $x\in X$ there exists a finite index congruence $\rho_x$ on $A$ that separates $a$ from $\langle x\rangle.$  Let $\rho=\bigcap_{x\in X}\rho_x$.  It is clear that $\rho$ has finite index.  Now consider any $b\in B.$  We have that $b\in\langle x\rangle$ for some $x\in X.$  Since $(a, b)\notin\rho_x,$ it follows that $(a, b)\notin\rho$.  Thus $\rho$ separates $a$ from $B,$ and hence $B$ is weakly subact separable. 
\end{proof}

Given a monoid $M$ and a congruence $\rho$ on $M,$ the following result characterises the separability conditions for the cyclic $M$-act $\textbf{M}/\rho$ in terms of separability in the monoid $M/\rho.$

\begin{thm}
\label{monoid quotient separability}
Let $M$ be a monoid and let $\rho$ be a congruence on $M.$ 
\begin{enumerate}[leftmargin=*]
\item The $M$-act $\textbf{M}/\rho$ is residually finite if and only if the monoid $M/\rho$ is residually finite (as a monoid).
\item The $M$-act $\textbf{M}/\rho$ is weakly subact separable if and only if the monoid $M/\rho$ satisfies the separability condition with respect to the collection of all principal right ideals.
\item The $M$-act $\textbf{M}/\rho$ is strongly subact separable if and only if the monoid $M/\rho$ satisfies the separability condition with respect to the collection of all right ideals.
\item The $M$-act $\textbf{M}/\rho$ is completely separable if and only if the monoid $M/\rho$ is completely separable.
\end{enumerate}
\end{thm}

\begin{proof}
We prove all four statements with a single argument.
Let $A$ denote the cyclic $M$-act $\textbf{M}/\rho$ and let $N$ denote the monoid $M/\rho.$
Let $\mathcal{S}$ be the collection associated with any of the four act separability conditions, and let $\mathcal{S}^{\prime}$ be the collection associated with the corresponding monoid separability condition.  (By Proposition \ref{WSS=CSS}, for weak subact separability we may assume that $\mathcal{S}$ is the collection of all cyclic subacts.)  We show that $A$ satisfies the separability condition with respect to $\mathcal{S}$ if and only if $N$ satisfies the separability condition with respect to $\mathcal{S}^{\prime}.$\par
($\Rightarrow$)  Suppose that $A$ satisfies the separability condition with respect to $\mathcal{S}.$  Let $X\subseteq N$ with $X\in\mathcal{S}^{\prime}$ and let $y\in N\!\setminus\!X.$  Then $X\in\mathcal{S}$ (e.g.\ if $X$ is a right ideal of $N,$ then for any $[m]_{\rho}\in X$ and $q\in M$ we have $[m]_{\rho}\!\cdot\!q=[mq]_{\rho}=[m]_{\rho}[q]_{\rho}\in X,$ so $X$ is a subact of $A$) and $y\in A\!\setminus\!X.$
Therefore, there exist a finite $M$-act $C$ and a surjective $M$-homomorphism $\theta : A\to C$ such that $y\theta\notin X\theta.$  Now, by Proposition \ref{rep}, we have a monoid homomorphism $\psi : M\to T_C$ given by $c(m\psi)=cm$ for all $c\in C$ and $m\in M.$  Notice that $T_C$ is finite since $C$ is finite.  Therefore, the congruence $\sigma=\text{ker }\psi$ on $M$ has finite index.  Using the fact that $\theta$ is surjective, we have
\begin{align*}
(m, n)\in\sigma&\iff m\psi=n\psi\\
&\iff c(m\psi)=c(n\psi)\text{ for all }c\in C\\
&\iff cm=cn\text{ for all }c\in C\\
&\iff([q]_{\rho}\theta)m=([q]_{\rho}\theta)n\text{ for all }q\in M\\
&\iff[qm]_{\rho}\theta=[qn]_{\rho}\theta\text{ for all }q\in M\\
&\iff([qm]_{\rho}, [qn]_{\rho})\in\text{ker }\theta\text{ for all }q\in M,
\end{align*}
so
$$\sigma=\{(m, n)\in M\times M : ([qm]_{\rho}, [qn]_{\rho})\in\text{ker }\theta\text{ for all }q\in M\}.$$
Now let $(m, n)\in\rho.$  Since $\rho$ is a congruence, for any $q\in M$ we have $[qm]_{\rho}=[qn]_{\rho},$ so certainly $([qm]_{\rho}, [qn]_{\rho})\in\text{ker }\theta.$  Thus $(m, n)\in\sigma,$ and hence $\rho\subseteq\sigma.$  The relation
$$\sigma/\rho=\{([m]_{\rho}, [n]_{\rho})\in N\times N : (m, n)\in\sigma\}$$
is a congruence on $N$ (by the Third Isomorphism Theorem).  Moreover, $\sigma/\rho$ has finite index since $\sigma$ has finite index.\par 
If $([m]_{\rho}, [n]_{\rho})\in\sigma/\rho,$ then $(m, n)\in\sigma,$ so $([m]_{\rho}, [n]_{\rho})=([1m]_{\rho}, [1n]_{\rho})\in\text{ker }\theta.$  Thus $\sigma/\rho\subseteq\text{ker }\theta.$  Since $\text{ker }\theta$ separates $y$ from $X,$ we deduce that $\sigma/\rho$ separates $y$ from $X.$  Hence, the monoid $N$ satisfies the separability condition with respect to $\mathcal{S}^{\prime}.$\par
($\Leftarrow$)  Conversely, suppose that $N$ satisfies the separability condition with respect to $\mathcal{S}^{\prime}.$  Let $X\subseteq A$ with $X\in\mathcal{S}$ and let $a\in A\!\setminus\!X.$  Then $X\in\mathcal{S}^{\prime}$ (e.g.\ if $X$ is a subact of $A,$ then for any $[m]_{\rho}\in X$ and $[q]_{\rho}\in N$ we have $[m]_{\rho}[q]_{\rho}=[mq]_{\rho}=[m]_{\rho}\!\cdot\!q\in X,$ so $X$ is a right ideal of $N$) and $a\in N\!\setminus\!X.$  Therefore, there exist a finite monoid $P$ and a monoid homomorphism $\phi : N\to P$ such that $a\phi\notin X\phi.$  It is easy to see that $P$ is an $M$-act via $p\ast m=p([m]_{\rho}\phi).$  For any $m, q\in M,$ we have 
$$([m]_{\rho}\phi)\ast q=([m]_{\rho}\phi)([q]_{\rho}\phi)=([m]_{\rho}[q]_{\rho})\phi=([mq]_{\rho})\phi=([m]_{\rho}\!\cdot\!q)\phi.$$
Thus $\phi$ may be viewed as an $M$-homomorphism from $A$ to $P,$ and we have that $\phi$ separates $a$ from $X.$  This completes the proof.
\end{proof}

From Theorem \ref{monoid quotient separability} we deduce several corollaries.

\begin{corollary}
\label{monoid separability}
Let $M$ be a monoid.
\begin{enumerate}[leftmargin=*]
\item The $M$-act $\textbf{M}$ is residually finite if and only if $M$ is residually finite.
\item The $M$-act $\textbf{M}$ is weakly subact separable if and only if $M$ satisfies the separability condition with respect to the collection of all principal right ideals.
\item The $M$-act $\textbf{M}$ is strongly subact separable if and only if $M$ satisfies the separability condition with respect to the collection of all right ideals.
\item The $M$-act $\textbf{M}$ is completely separable if and only if $M$ is completely separable.
\end{enumerate}
\end{corollary}

\begin{corollary}
\label{group}
Let $G$ be a group.
\begin{enumerate}[leftmargin=*]
\item The $G$-act $\textbf{G}$ is residually finite if and only if $G$ is residually finite (as a group).
\item The $G$-act $\textbf{G}$ {\em is} strongly subact separable (and hence weakly subact separable).
\item The $G$-act $\textbf{G}$ is completely separable if and only if it is finite.
\end{enumerate}
\end{corollary}

\begin{proof}
(1)  This follows from Corollary \ref{monoid separability}(1) and the fact that a group is residually finite as a group if and only if it is residually finite as a monoid.\par
(2)  Since $G$ has no proper right ideals, it trivially satisfies the separability condition with respect to the collection of all right ideals, so $\textbf{G}$ is strongly subact separable by Corollary \ref{monoid separability}(3).\par
(3)  This follows from Corollary \ref{monoid separability}(4) and \cite[Lemma 2.4]{Miller}.
\end{proof}

\begin{remark}
For various classes of algebraic structures, including groups and semigroups, cyclic subalgebra separability implies residual finiteness.  Corollary \ref{group} shows that this is not the case for acts; indeed, an act can be strongly subact separable but not residually finite.
\end{remark}

\begin{corollary}
\label{all quotients RF}
Let $M$ be a monoid.  If every cyclic $M$-act is residually finite (resp.\ completely separable), then every quotient of $M$ is residually finite (resp.\ completely separable).
\end{corollary}

\begin{proof}
Let $\rho$ be a congruence on $M.$  Since the cyclic $M$-act $\textbf{M}/\rho$ is residually finite (resp.\ completely separable), the quotient $M/\rho$ is residually finite (resp.\ completely separable) by Theorem \ref{monoid quotient separability}.
\end{proof}

An $M$-act $A$ is said to be {\em faithful} if for any pair of distinct elements $m, n\in M$ there exists some $a\in A$ such that $am\neq an.$  For instance, the $M$-act $\textbf{M}$ is faithful.\par
The following result enhances Corollary \ref{monoid separability}(1).

\begin{prop}
The following are equivalent for a monoid $M$:
\begin{enumerate}
\item the $M$-act $\textbf{M}$ is residually finite;
\item there exists a residually finite faithful $M$-act;
\item $M$ is residually finite (as a monoid).
\end{enumerate}
\end{prop}

\begin{proof}
(1)\,$\Rightarrow$\,(2) is obvious, and (1)\,$\Leftrightarrow$\,(3) is Corollary \ref{monoid separability}(1).\par
(2)\,$\Rightarrow$\,(3).  Let $m, n\in M$ with $m\neq n.$  By assumption, there exists a residually finite faithful $M$-act $A.$  Then there exists some $a\in A$ such that $am\neq an.$  Since $A$ is residually finite, there exist a finite $M$-act $C$ and an $M$-homomorphism $\theta : A\to C$ such that $(am)\theta\neq(an)\theta.$  By Proposition \ref{rep}, we have a monoid homomorphism $\psi : M\to T_C$ given by $c(p\psi)=cp$ for all $c\in C$ and $p\in M.$  The monoid $T_C$ is finite since $C$ is finite.  Moreover, we have that
$$(a\theta)(m\psi)=(a\theta)m=(am)\theta\neq(an)\theta=(a\theta)n=(a\theta)(n\psi),$$
so $m\psi\neq n\psi.$  Thus $M$ is residually finite.
\end{proof}

We now present some examples to show that, for the class of monoid acts, the converses of the statements in Proposition \ref{algebra separability} do not hold.

\begin{ex}
\label{SSS, not CS}
\textit{There exists a monoid $M$ such that $\textbf{M}$ is strongly subact separable but not completely separable}.\par
Let $M$ be the monoid with zero defined by the presentation 
$$\langle a, b, c\,|\,w^2=0\,(w\in\{a, b, c\}^+)\rangle.$$
It was proved in \cite[Example 5.11]{Miller} that $M$ is not completely separable, and hence $\textbf{M}$ is not completely separable by Corollary \ref{monoid separability}.\par
To prove that $\textbf{M}$ is strongly subact separable, let $I$ be a right ideal of $M$ and let $u\in M\!\setminus\!I.$  Clearly $0\in I,$ so $u\neq 0.$  Let $n=|u|$ be the length of $u$ over $\{a, b, c\}.$  Let $J$ be the ideal 
$$\{w\in\{a, b, c\}^+ : |w|\geq n+1\}\cup\{0\}$$
of $M.$  It is clear that the Rees quotient $M/J$ is finite and $[u]_J=\{u\}.$  Thus the Rees congruence $\sim_J$ on $M$ separates $u$ from $I,$ as required.
\end{ex}

\begin{ex}
\label{WSS, not SSS}
\textit{There exists a monoid $M$ such that $\textbf{M}$ is weakly subact separable but not strongly subact separable}.\par
Let $\mathbb{N}$ be the semigroup of natural numbers under addition, let $G$ be any infinite residually finite group, and let $M$ be the monoid $(\mathbb{N}\times G)^1.$\par
We first show that $\textbf{M}$ is not strongly subact separable.  Let $e$ be the identity of $G,$ and consider the (right) ideal 
$$A=\bigl(\{2\}\times(G\!\setminus\!\{e\}\bigr)\cup(\{3, 4, \dots\}\times G)$$ 
of $M.$  We claim that $(2, e)$ cannot be separated from $A.$  Indeed, let $\sim$ be a finite index congruence on $M.$  Then there exist $g, h\in G$ with $g\neq h$ such that $(1, g)\sim(1, h).$  Then we have
$$(2, e)=(1, g)(1, g^{-1})\sim(1, h)(1, g^{-1})=(2, hg^{-1})\in A,$$
as required.\par
We prove that $\textbf{M}$ is weakly subact separable by showing that $M$ satisfies the separability condition with respect to the collection of all principal right ideals and invoking Corollary \ref{monoid separability}.  So, let $A=\langle a\rangle,$ where $a=(n, g),$ be a principal right ideal of $M,$ and let $b\in M\!\setminus\!A.$  It can be easily shown that $$A=\{(n, g)\}\cup(\{n+1, n+2, \dots\}\times G).$$
If $b=1,$ then $b$ can be separated from $A$ by the congruence with classes $\{1\}$ and $M\!\setminus\!\{1\}.$  Suppose then that $b=(m, h)\in\mathbb{N}\times G.$
Since $b\notin A,$ we must have that $m\leq n.$  Let $J$ be the ideal $\{m+1, m+2, \dots\}$ of $\mathbb{N}.$  We now define a monoid homomorphism
$$\theta : M\to(\mathbb{N}/J)^1, 1\mapsto 1, (p, u)\mapsto[p]_J.$$
Let $0$ denote the element $J$ of the Rees quotient $\mathbb{N}/J.$  Clearly $b\theta=\{m\}$ and $(p, u)\theta=0$ for all $p\geq m+1.$  In particular, we have $(A\!\setminus\{a\})\theta=\{0\}.$
Furthermore, if $m<n$ we have $A\theta=\{0\},$ so $\theta$ separates $b$ from $A,$ and we are done.\par
Suppose then that $m=n.$  Then $g\neq h.$  Since $G$ is residually finite, there exist a finite group $K$ and a homomorphism $\phi : G\to K$ such that $g\phi\neq h\phi.$  We define a homomorphism
$$\psi : M\to\bigl((\mathbb{N}/J)\times K\bigr)^1, 1\mapsto 1, (p, u)\mapsto([p]_J, u\phi).$$  
Clearly $\psi$ separates $b$ from $A,$ as desired.
\end{ex}

\begin{ex}
\label{RF, not WSS}
\textit{There exists a monoid $M$ such that $\textbf{M}$ is residually finite but not weakly subact separable (and hence not completely separable)}.\par
Let $A=\{a^i : i\geq 0\}\cong\mathbb{N}_0,$ let $B=\{b_i : i\in\mathbb{Z}\}\cup\{0\}$ be a null semigroup, and let $M=A\cup B.$  We define a multiplication on $M,$ extending those on $A$ and $B,$ as follows:
$$a^ib_j=b_ja^i=b_{i+j},\, a^i0=0a^i=0.$$
It is to see that $M$ is a commutative monoid under this multiplication.\par
We first prove that $\textbf{M}$ is not weakly subact separable by showing that we cannot separate the element $b_0$ from the cyclic subact 
$$C=\langle b_1\rangle=\{b_i : i\in\mathbb{N}\}\cup\{0\}.$$
Indeed, let $\sim$ be a finite index congruence on $\textbf{M}.$  Then there exist $i, j\in\mathbb{N}$ with $i<j$ such that $b_{-i}\sim b_{-j}.$  It follows that
$$b_0=b_{-j}a^j\sim b_{-i}a^j=b_{j-i}\in C,$$ as required.\par
To prove that $\textbf{M}$ is residually finite, we just need to show that $M$ is residually finite by Corollary \ref{monoid separability}.  So, let $x, y\in M$ with $x\neq y.$  There are three cases to consider.\par
(1) $x\in A.$  Then $x=a^i$ for some $i\geq 0.$  We may assume without loss of generality that $y$ belongs to the ideal $I=M\!\setminus\!\{1, a, \dots, a^i\}.$  Then the Rees congruence $\sim_I$ separates $x$ and $y.$\par
(2) $x\in B\!\setminus\!\{0\}$ and $y=0.$  In this case the congruence with classes $A, B\!\setminus\!\{0\}$ and $\{0\}$ separates $x$ and $y.$\par
(3) $x, y\in B\!\setminus\!\{0\}.$  Then $x=b_i$ and $y=b_j$ for some $i, j\in\mathbb{Z}$ with $i\neq j.$  Let $n=|i-j|+1,$ and let $\sim$ be the equivalence relation on $M$ with classes $\{0\}, C_m, D_m \,(m\in\{0, \dots, n-1\}),$ where
$$C_m=\{a^k : k\equiv m\:\:\text{mod }n\},\; D_m=\{b_k : k\equiv m\:\:\text{mod }n\}.$$
It is easy to see that $\sim$ is a finite index congruence on $M$ that separates $x$ and $y,$ as required.
\end{ex}

We have observed that, in general, weakly subact separable monoid acts need not be residually finite.  The following lemma provides a condition under which weak subact separability implies residual finiteness.

\begin{lemma}
\label{Green}
Let $M$ be a monoid, let $A$ be an $M$-act, and suppose that Green's relation $\mathcal{R}_A$ is the equality relation on $A$ (or, equivalently, the preorder $\leq_A$ is antisymmetric).  If $A$ is weakly subact separable then it is residually finite.
\end{lemma}

\begin{proof}
Let $a, b\in A$ with $a\neq b.$  Then $(a, b)\notin\mathcal{R}_A,$ so $\langle a\rangle\neq\langle b\rangle.$  Letting $B=\langle b\rangle,$ we may assume without loss of generality that $a\notin B.$  Since $A$ is weakly subact separable, there exist a finite $M$-act $C$ and an $M$-homomorphism $\theta : A\to C$ such that $a\theta\notin B\theta.$  In particular, we have $a\theta\neq b\theta,$ as required.
\end{proof}

\begin{remark}
\label{Green remark}
If $M$ is a commutative idempotent monoid, then for every $M$-act $A$ Green's relation $\mathcal{R}_A$ is the equality relation on $A.$  Indeed, if $(a, b)\in\mathcal{R}_A$ then there exist $m, n\in M$ such that $a=bm$ and $b=an,$ and hence  
$$a=anm=amn=bm^2n=bmn=an=b,$$
as required.
\end{remark}

The following result shows that each of the four separability conditions is preserved under arbitrary disjoint unions.

\begin{prop}
\label{disjoint union}
Let $M$ be a monoid, and let $A$ be an $M$-act that is a disjoint union of subacts $A_i, i\in I.$  Let $\mathcal{C}$ be any of the four separability conditions.  Then $A$ satisfies $\mathcal{C}$ if and only if each $A_i$ satisfies $\mathcal{C}.$
\end{prop}

\begin{proof}
The forward implication follows from Lemma \ref{subalgebra}, so we just need to prove the converse.\par
Let $X\subseteq A$ be of the type associated with $\mathcal{C}$ and $a\in A\!\setminus\!X.$  Now, $a\in A_i$ for some $i\in I.$  Clearly $A$ is the disjoint union of its subacts $A_i$ and $A\!\setminus\!A_i.$
There are two cases to consider.\par
(1) $X\cap A_i=\emptyset.$  Clearly the equivalence relation on $A$ with classes $A_i$ and $A\!\setminus\!A_i$ is a finite index congruence that separates $a$ from $X.$\par
(2) $X\cap A_i\neq\emptyset.$  It is easy to see that $X\cap A_i$ is of the type associated with $\mathcal{C}.$  Therefore, since $A_i$ satisfies $\mathcal{C},$ there exist a finite $M$-act $C$ and an $M$-homomorphism $\phi : A_i\to C$ such that $a\phi\notin(X\cap A_i)\phi.$  Now define an $M$-homomorphism $\theta : A\to C\cup\{0\},$ where $0$ is a zero element disjoint from $C,$ by 
$$a^{\prime}\theta=
\begin{cases}
a^{\prime}\phi&\text{if }a^{\prime}\in A_i\\ 
0&\text{otherwise.}
\end{cases}$$
Clearly $\theta$ separates $a$ from $X,$ as required.
\end{proof}

An $M$-act $A$ is said to be {\em decomposable} if it is the disjoint union of two subacts $A_1$ and $A_2$; otherwise, $A$ is {\em indecomposable}.  Every $M$-act has a unique decomposition into indecomposable subacts \cite[Theorem 1.5.10]{Kilp}.  From Proposition \ref{disjoint union} we deduce:

\begin{corollary}
\label{decomposition}
Let $M$ be a monoid, let $A$ be an $M$-act, and let $A=\bigcup_{i\in I}A_i$ be the unique decomposition of $A$ into indecomposable subacts.  Let $\mathcal{C}$ be any of the four separability conditions.  Then $A$ satisfies $\mathcal{C}$ if and only if each $A_i$ satisfies $\mathcal{C}.$
\end{corollary}

An $M$-act $A$ is said to be {\em free} if it has a generating set $U$ such that every $a\in A$ can be uniquely written in the form $a=um$ for some $u\in U$ and $m\in M.$  By \cite[Theorem 1.5.13]{Kilp}, an $M$-act $A$ is free if and only if it is $M$-isomorphic to a disjoint union of $M$-acts all of which are $M$-isomorphic to $\textbf{M}.$  Thus, by Proposition \ref{disjoint union}, we have:

\begin{corollary}
Let $M$ be a monoid, let $A$ be a free $M$-act, and let $\mathcal{C}$ be any of the four separability conditions.  Then $A$ satisfies $\mathcal{C}$ if and only if $\textbf{M}$ satisfies $\mathcal{C}.$
\end{corollary}

We now provide a characterisation of completely separable group acts.

\begin{prop}
Let $G$ be a group and let $A$ be a $G$-act.  Then $A$ is completely separable if and only if it is a disjoint union of finite cyclic $G$-acts.
\end{prop}

\begin{proof}
Since finite acts are completely separable, the reverse implication follows from Proposition \ref{disjoint union}.\par
For the direct implication, let $A=\bigcup_{i\in I}A_i$ be the unique decomposition of $A$ into indecomposable subacts.  Then each $A_i$ is cyclic by \cite[Proposition 2.3]{Knauer}.  We claim that each $A_i$ is finite.  Assume for a contradiction that some $A_i=\langle x\rangle$ is infinite.  Let $\sim$ be a finite index congruence on $A_i.$  Then there exist some $g, h\in G$ such that $xg\neq xh$ and $xg\sim xh.$  It follows that
$$x=(xg)g^{-1}\sim(xh)g^{-1}=x(hg^{-1})\in A_i\!\setminus\!\{x\},$$
so $A_i$ is not completely separable.  But then, by Corollary \ref{decomposition}, this contradicts the fact that $A$ is completely separable.
\end{proof}

\section{\large{Equivalent Formulations}\nopunct}
\label{sec:EF}

The purpose of this section is to provide equivalent formulations of the four separability conditions.\par
It is a well-known fact in universal algebra that residual finiteness is equivalent to the intersection of all finite index congruences equalling the equality relation.  Thus we have:

\begin{prop}
\label{RF characterisation}
An $M$-act $A$ is residually finite if and only if the intersection of all finite index congruences on $A$ is the equality relation $\{(a, a) : a\in A\}.$
\end{prop}

The above characterisation leads to the following necessary and sufficient condition for a quotient of an $M$-act to be residually finite.

\begin{prop}
\label{quotient}
Let $M$ be a monoid, let $A$ be an $M$-act, and let $\rho$ be a congruence on $A.$  Then $A/\rho$ is residually finite if and only if $\rho$ is the intersection of a family of finite index congruences on $A.$
\end{prop}

\begin{proof}
($\Rightarrow$) Let $\{\sigma_i : i\in I\}$ be the set of all finite index congruences on $A/\rho.$  By Proposition \ref{RF characterisation}, the intersection $\cap_{i\in I}\sigma_i$ is the equality relation on $A/\rho.$ 
For each $i\in I$ we define a congruence $\rho_i$ on $A$ by
$$a\,\rho_i\,b\iff[a]_{\rho}\,\sigma_i\,[b]_{\rho}.$$
Clearly each $\rho_i$ has finite index since $\sigma_i$ has finite index.  We claim that $\rho=\cap_{i\in I}\rho_i.$  That $\rho\subseteq\cap_{i\in I}\rho_i$ follows from the fact that each $\sigma_i$ is reflexive.  For the reverse containment, let $(a, b)\notin\rho.$  Then $([a]_{\rho}, [b]_{\rho})\notin\cap_{i\in I}\sigma_i.$  Therefore, there exists some $i\in I$ such that $([a]_{\rho}, [b]_{\rho})\notin\sigma_i,$ and hence $(a, b)\notin\rho_i.$  Thus $(a, b)\notin\cap_{i\in I}\rho_i,$ as required.\par
($\Leftarrow$) Let $\rho=\cap_{i\in I}\rho_i$ where each $\rho_i$ is a finite index congruence on $A.$  Let $x, y$ be distinct elements of $A/\rho.$  Now, $x=[a]_{\rho}$ and $y=[b]_{\rho}$ for some $a, b\in A.$  Since $(a, b)\notin\rho,$  there exists some $i\in I$ such that $(a, b)\notin\rho_i.$
The relation
$$\rho_i/\rho=\{([c]_{\rho}, [d]_{\rho})\in A/\rho\times A/\rho : (c, d)\in\rho_i\}$$
is a congruence on $A/\rho,$ it has finite index since $\rho_i$ has finite index, and $(x, y)\notin\rho_i/\rho$ since $(a, b)\notin\rho_i.$  Thus $\rho_i/\rho$ separates $x$ and $y,$ and hence $A/\rho$ is residually finite.
\end{proof}

Since the congruences on the $M$-act $\textbf{M}$ are precisely the right congruences on $M,$ we obtain:

\begin{corollary}
\label{cyclic RF}
Let $M$ be a monoid and let $A$ be a cyclic $M$-act, so $A\cong\textbf{M}/\rho$ for some right congruence $\rho$ on $M.$  Then $A$ is residually finite if and only if $\rho$ is the intersection of a family of finite index right congruences on $M.$
\end{corollary}

The next two results provide equivalent characterisations of weak subact separability and strong subact separability in terms of separation of pairs of elements.

\begin{prop}
\label{WSS condition}
Let $M$ be a monoid.   The following are equivalent for an $M$-act $A$:
\begin{enumerate}
\item $A$ is weakly subact separable;
\item for any cyclic subact $B$ of $A$ and any $a\in A\!\setminus\!B,$ there exist a finite $M$-act $C$ and an $M$-homomorphism $\theta : A/B\to C$ such that $a\theta\neq 0_B\theta.$
\end{enumerate}
\end{prop}

\begin{proof}
(1)\,$\Rightarrow$\,(2).  Let $B$ be a cyclic subact of $A$ and let $a\in A\!\setminus\!B.$  Since $A$ is weakly subact separable, there exist a finite $M$-act $D$ and an $M$-homomorphism $\phi : A\to D$ such that $a\phi\notin B\phi.$  Let $E$ denote the subact $B\phi$ of $D,$ and let $$B^{\prime}=\{a^{\prime}\in A : a^{\prime}\phi\in E\}.$$
It can be easily shown that $B^{\prime}$ is a subact of $A,$ and clearly $B\subseteq B^{\prime}.$  Letting $C=D/E,$ we define an $M$-homomorphism 
$$\theta : A/B\to C, a^{\prime}\mapsto
\begin{cases}
a^{\prime}\phi&\text{if }a^{\prime}\in A\!\setminus\!B^{\prime}\\ 
0_E&\text{if }a^{\prime}\in(B^{\prime}\!\setminus\!B)\cup\{0_B\}.
\end{cases}$$
We have that $a\theta=a\phi\in C\!\setminus\!\{0_E\}$ and $0_B\theta=0_E,$ so $a\theta\neq 0_B\theta,$ as required.\par
(2)\,$\Rightarrow$\,(1).  By Proposition \ref{WSS=CSS}, it suffices to prove that $A$ is cyclic subact separable.  So, let $B$ be a cyclic subact of $A$ and let $a\in A\!\setminus\!B.$  By assumption, there exist a finite $M$-act $C$ and an $M$-homomorphism $\theta : A/B\to C$ such that $a\theta\neq 0_B\theta.$  Letting $\phi$ denote the canonical $M$-homomorphism $A\to A/B,$ we conclude that $\phi\circ\theta : A\to C$ separates $a$ from $B.$  Thus $A$ is weakly subact separable.
\end{proof}

\begin{prop}
\label{SSS condition}
Let $M$ be a monoid.   The following are equivalent for an $M$-act $A$:
\begin{enumerate}
\item $A$ is strongly subact separable;
\item for any subact $B$ of $A$ and any $a\in A\!\setminus\!B,$ there exist a finite $M$-act $C$ and an $M$-homomorphism $\theta : A/B\to C$ such that $a\theta\neq 0_B\theta.$
\end{enumerate}
\end{prop}

\begin{proof}
The proof is essentially the same as that of Proposition \ref{WSS condition}.
\end{proof}

In the following we present a necessary and sufficient condition for a monoid act to be completely separable.  We first make the following definition.

\begin{defn}
\label{CS defn}
Let $M$ be a monoid and let $A$ be an $M$-act.  For each pair $a, b\in A,$ define a set
$$[a, b]_M=\{m\in M : a=bm\}.$$
In the case that $A=\textbf{M},$ we omit the subscript $M$ in the above notation and just write $[a, b].$
\end{defn}

\begin{thm}
\label{CS condition}
Let $M$ be a monoid.  An $M$-act $A$ is completely separable if and only if for each $a\in A$ the set $\{[a, b]_M : b\in A\}$ is finite.
\end{thm}

\begin{proof}
(1)\,$\Rightarrow$\,(2).  We prove the contrapositive.  Suppose that for some $a\in A$ the set $\{[a, b]_M : b\in A\}$ is infinite.  Let $\sim$ be a finite index congruence on $A.$  Then there exist some $b, c\in A$ such that $b\sim c$ and $[a, b]_M\neq[a, c]_M.$  Therefore, there exists some $m\in M$ such that 
$$m\in([a, b]_M\cup[a, c]_M)\!\setminus\!([a, b]_M\cap[a, c]_M).$$  
We may assume without loss of generality that $m\in [a, b]_M\!\setminus\![a, c]_M.$  It follows that
$$a=bm\sim cm\neq a.$$
Hence, we cannot separate $a$ from $A\!\setminus\!\{a\}$ in a finite index congruence, so $A$ is not completely separable.\par
(2)\,$\Rightarrow$\,(1).  Let $a\in A.$  Define a relation $\sim_a$ on $A$ by $$b\sim_ac\iff[a, b]_M=[a, c]_M.$$
Clearly $\sim_a$ is an equivalence relation, and it has finite index since the set $\{[a, b]_M : b\in A\}$ is finite.  Let $b\sim_ac$ and $m\in M.$  Then $[a, b]_M=[a, c]_M.$  For each $n\in M$ we have 
\begin{align*}
n\in [a, bm]_M&\iff a=(bm)n=b(mn)\iff mn\in[a, b]_M=[a, c]_M\\&\iff a=c(mn)=(cm)n\\&\iff n\in[a, cm]_M,
\end{align*}
so $[a, bm]_M=[a, cm]_M$ and hence $bm\sim_acm.$  Thus $\sim_a$ is a congruence on $A.$  Now, observe that $1\in[a, b]_M$ if and only if $a=b,$ so $[a]_{\sim_a}=\{a\}.$  Since $a$ was chosen arbitrarily, we have proven that $A$ is completely separable.
\end{proof}

Kozhukhov proved that every act over a finite monoid is residually finite \cite[Proposition 3]{Kozhukhov1}.  Since a finite monoid has only finitely many subsets, we deduce from Theorem \ref{CS condition} that every act over a finite monoid is in fact completely separable (and hence satisfies the other three separability conditions).

\begin{corollary}
\label{finite monoid}
Let $M$ be a finite monoid.  Then every $M$-act is completely separable.
\end{corollary}

From Corollaries \ref{group}(3) and \ref{finite monoid}, we obtain:

\begin{corollary}
\label{all G-acts CS}
Let $G$ be a group.  Then every $G$-act is completely separable if and only if $G$ is finite.
\end{corollary}

Theorem \ref{CS condition}, together with Corollary \ref{monoid separability}(4), yields the following characterisation of complete separability for monoids.

\begin{corollary}
\label{CS monoid}
A monoid $M$ is completely separable if and only if for each $a\in M$ the set $\{[a, b] : b\in M\}$ is finite.
\end{corollary}

\section{\large{Finiteness Conditions on Monoids}\nopunct}
\label{sec:FC}

A {\em finiteness condition} for a class of universal algebras is a property that is satisfied by at least all finite members of that class.  In the remainder of the paper we study several natural finiteness conditions on monoids arising from separability in their acts.  Specifically, given any of the four separability conditions $\mathcal{C},$ we investigate which monoids have the property that all their acts satisfy $\mathcal{C},$ and which monoids satisfy the weaker condition that all their finitely generated acts satisfy $\mathcal{C}.$  The fact that these are finiteness conditions follows from Corollary \ref{finite monoid}.\par
It turns out that the property that all finitely generated acts satisfy $\mathcal{C}$ is in fact equivalent to the property that all {\em cyclic} acts satisfy $\mathcal{C}.$

\begin{thm}
\label{all cyclic acts}
Let $\mathcal{C}$ be any of the four separability conditions.  For a monoid $M,$ all finitely generated $M$-acts satisfy $\mathcal{C}$ if and only if all cyclic $M$-acts satisfy $\mathcal{C}.$
\end{thm}

\begin{proof}
The direct implication is obvious, so we just need to prove the converse.  We prove by induction on the number of generators.  Let $A$ be an $M$-act with a finite generating $X$ where $|X|>1,$ and assume that all $M$-acts with a smaller generating set satisfy $\mathcal{C}.$  Clearly we may assume that $X$ is a minimal generating set for $A.$  If there exists $x\in X$ such that $A\!\setminus\!\langle x\rangle$ is  a subact, then $\langle x\rangle$ and $A\!\setminus\!\langle x\rangle=\langle X\!\setminus\!\{x\}\rangle$ satisfy $\mathcal{C}$ by the inductive hypothesis, and hence $A=\langle x\rangle\cup(A\!\setminus\!\langle x\rangle)$ satisfies $\mathcal{C}$ by Proposition \ref{disjoint union}.  We may assume then that for each $x\in X$ there exists $y\in X$ such that $\langle x\rangle\cap\langle y\rangle\neq\emptyset.$\par
Now, let $U\subseteq A$ be of the type associated with $\mathcal{C}$ and let $a\in A\!\setminus\!U.$  (For residual finiteness, we assume that $U$ is a finite subset.  For weak subact separability and strong subact separability, by Propositions \ref{WSS condition} and \ref{SSS condition}, respectively, it suffices to consider the case that $U$ is a trivial subact.)  There are three cases to consider.\par 
(1) Suppose first that there exists $x\in X$ such that $a\notin\langle x\rangle.$  Let $B=\langle x\rangle.$  Then the Rees quotient $A/B=\langle X\!\setminus\!\{x\}\rangle$ satisfies $\mathcal{C}$ by the inductive hypothesis.  Clearly $U^{\prime}=(U\setminus\!B)\cup\{0_B\}\subseteq A/B$ is of the type associated with $\mathcal{C}.$  Therefore, there exist a finite $M$-act $C$ and an $M$-homomorphism $\theta : A/B\to C$ such that $a\theta\notin U^{\prime}\theta.$  Letting $\phi : A\to A/B$ be the canonical $M$-homomorphism, we have that $\phi\circ\theta : A\to C$ separates $a$ from $U.$\par
(2) Now assume that $a\in\langle x\rangle$ for each $x\in X,$ and suppose that $U\cap\langle y\rangle=\emptyset$ for some $y\in X.$  Let $B=\langle y\rangle.$  Then $a\in B.$  The Rees quotient $A/B=\langle X\!\setminus\!\{y\}\rangle$ satisfies $\mathcal{C}$ by the inductive hypothesis.  Therefore, there exist a finite $M$-act $C$ and an $M$-homomorphism $\theta : A/B\to C$ such that $0_B\theta\notin U\theta.$  Again, letting $\phi : A\to A/B$ be the canonical $M$-homomorphism, we have that $\phi\circ\theta : A\to C$ separates $a$ from $U.$\par
(3) Finally, assume that $a\in\langle x\rangle$ and $U\cap\langle x\rangle\neq\emptyset$ for each $x\in X.$  Fix an $x\in X,$ let $B=\langle x\rangle,$ and let $V=U\cap B.$  Clearly $V$ is of the type associated with $\mathcal{C}.$  Since $B$ is cyclic, it satisfies $\mathcal{C}$ by assumption.  Therefore, there exist a finite $M$-act $C$ and an $M$-homomorphism $\theta : B\to C$ such that $a\theta\notin V\theta.$  Let $C^{\prime}=C\cup(A\!\setminus\!B).$  Extend the action of $M$ on $C$ to $C^{\prime}$ by setting
$$a^{\prime}\cdot m=
\begin{cases}
a^{\prime}m&\text{if }a^{\prime}m\in A\!\setminus\!B,\\ 
(a^{\prime}m)\theta&\text{if }a^{\prime}m\in B
\end{cases}$$
for all $a^{\prime}\in A\!\setminus\!B$ and $m\in M.$  It can be easily checked that $C^{\prime}$ is an $M$-act under this action.  We now define an $M$-homomorphism 
$$\theta^{\prime} : A\to C^{\prime}, a^{\prime}\mapsto
\begin{cases}
a^{\prime}&\text{if }a^{\prime}\in A\!\setminus\!B,\\ 
a^{\prime}\theta&\text{if }a^{\prime}\in B.
\end{cases}$$
Let $B^{\prime}$ be the subact of $C^{\prime}$ generated by $(X\!\setminus\!\{x\})\theta^{\prime}.$  Then $B^{\prime}$ satisfies $\mathcal{C}$ by the inductive hypothesis.  We have that $a\theta\in B^{\prime}$ and $U\theta\cap B^{\prime}\neq\emptyset.$  Let $U^{\prime}=U\theta\cap B^{\prime}.$  Then $U^{\prime}$ is of the type associated with $\mathcal{C}.$  Therefore, there exist a finite $M$-act $D$ and an $M$-homomorphism $\phi : B^{\prime}\to D$ such that $(a\theta)\phi\notin U^{\prime}\phi.$  In a similar way as above, let $D^{\prime}=D\cup(C^{\prime}\!\setminus\!B^{\prime})$ and extend the action of $M$ on $D$ to $D^{\prime}$ by setting
$$c^{\prime}\ast m=
\begin{cases}
c^{\prime}m&\text{if }c^{\prime}m\in C^{\prime}\!\setminus\!B^{\prime},\\ 
(c^{\prime}m)\phi&\text{if }c^{\prime}m\in B^{\prime}
\end{cases}$$
for all $c^{\prime}\in C^{\prime}\!\setminus\!B^{\prime}$ and $m\in M.$  
We define an $M$-homomorphism 
$$\phi^{\prime} : C^{\prime}\to D^{\prime}, c^{\prime}\mapsto
\begin{cases}
c^{\prime}&\text{if }a^{\prime}\in C^{\prime}\!\setminus\!B^{\prime},\\ 
c^{\prime}\phi&\text{if }c^{\prime}\in B^{\prime}.
\end{cases}$$
Now let $\psi=\theta^{\prime}\circ\phi^{\prime} : A\to D^{\prime}.$  We have that $D^{\prime}$ is finite since $D$ and $C^{\prime}\!\setminus\!B^{\prime}\subseteq C$ are finite.  We have that 
$$a\psi=(a\theta^{\prime})\phi^{\prime}=(a\theta)\phi\in D.$$
Consider $u\in U.$  Suppose first that $u\in V.$  Then $a\theta\neq u\theta.$  If $u\theta\notin U^{\prime},$ then $u\theta\in C^{\prime}\!\setminus\!B^{\prime}$ and $u\psi=(u\theta)\phi^{\prime}=u\theta\in C^{\prime}\!\setminus\!B^{\prime},$ so $a\psi\neq u\psi.$  If $u\theta\in U^{\prime},$ then $$u\psi=(u\theta)\phi\neq(a\theta)\phi=a\psi.$$  Now suppose that $u\in U\!\setminus\!V.$  Then $u=u\theta^{\prime}\in B^{\prime},$ so $u\in U^{\prime}.$  Then we have that
$$u\psi=(u\theta^{\prime})\phi^{\prime}=u\phi^{\prime}=u\phi\neq(a\theta)\phi=a\psi.$$
Since $u$ was chosen arbitrarily, we conclude that $a\psi\notin U\psi.$  This completes the proof.
\end{proof}

Theorem \ref{all cyclic acts} (for residual finiteness) and Corollary \ref{cyclic RF} together yield:

\begin{corollary}
\label{all fg acts RF}
The following are equivalent for a monoid $M$:
\begin{enumerate}
\item all finitely generated $M$-acts are residually finite;
\item all cyclic $M$-acts are residually finite;
\item every right congruence on $M$ is the intersection of a family of finite index right congruences on $M.$
\end{enumerate}
\end{corollary}

\begin{corollary}
\label{fg acts, quotients}
Let $M$ be a monoid such that every right congruence on $M$ is a (two-sided) congruence.  Then the following are equivalent:
\begin{enumerate}
\item all finitely generated $M$-acts are residually finite;
\item every congruence on $M$ is the intersection of a family of finite index right congruences on $M$;
\item every quotient of $M$ is residually finite.
\end{enumerate}
\end{corollary}

\begin{proof}
(1)\,$\Leftrightarrow$\,(2) follows immediately from Corollary \ref{all fg acts RF}.
The implication (1)\,$\Rightarrow$\,(3) follows from Corollary \ref{all quotients RF}.
 For (3)\,$\Rightarrow$\,(1), let $A$ be any cyclic $M$-act.  By Proposition \ref{cyclic act, right congruence} and the assumption that every right congruence on $M$ is a congruence, we have that $A\cong\textbf{M}/\rho$ for some congruence $\rho$ on $M.$  Since $M/\rho$ is residually finite, it follows from Theorem \ref{monoid quotient separability} that $A$ is residually finite.  Hence, by Theorem \ref{all cyclic acts}, all finitely generated $M$-acts are residually finite.
\end{proof}

Similarly, from Theorems \ref{monoid quotient separability} and \ref{all cyclic acts} we deduce:

\begin{corollary}
\label{all fg acts CS}
Let $M$ be a monoid such that every right congruence on $M$ is a (two-sided) congruence.  Then the following are equivalent:
\begin{enumerate}  
\item all finitely generated $M$-acts are completely separable;
\item all cyclic $M$-acts are completely separable;
\item every quotient of $M$ is completely separable.
\end{enumerate}
\end{corollary}

We now present several equivalent characterisations of the property that all acts are strongly subact separable; in particular, we show that this is equivalent to all acts being {\em weakly} subact separable.

\begin{prop}
\label{all acts SSS}
The following are equivalent for a monoid $M$:
\begin{enumerate}
\item all $M$-acts are strongly subact separable;
\item all $M$-acts are weakly subact separable;
\item all $M$-acts satisfy the separability condition with respect to the collection of all finite subacts;
\item all $M$-acts satisfy the separability condition with respect to the collection of all trivial subacts;
\item for any $M$-act $A$ containing zeroes, any zero $0\in A$ and any $a\in A\!\setminus\!\{0\},$ there exist a finite $M$-act $C$ and an $M$-homomorphism $\theta : A\to C$ such that $a\theta\neq 0\theta.$
\end{enumerate}
\end{prop}

\begin{proof}
The implications (1)\,$\Rightarrow$\,(2), (2)\,$\Rightarrow$\,(3), (3)\,$\Rightarrow$\,(4) and (4)\,$\Rightarrow$\,(5) are obvious.\par
(5)\,$\Rightarrow$\,(1).  Let $A$ be an $M$-act, let $B$ be a subact of $A,$ and let $a\in A\!\setminus\!B.$  Then $a\neq 0_B$ in $A/B.$  Therefore, by assumption, there exist a finite $M$-act $C$ and an $M$-homomorphism $\theta : A/B\to C$ such that $a\theta\neq 0_B\theta.$  Hence, by Proposition \ref{SSS condition}, the $M$-act $A$ is strongly subact separable.
\end{proof}

If a monoid $M$ has the property that all its acts are residually finite, then it clearly satisfies condition (5) of Proposition \ref{all acts SSS}, so we deduce:

\begin{corollary}
\label{all RF implies all SSS}
Let $M$ be a monoid.  If all $M$-acts are residually finite, then all $M$-acts are strongly subact separable (and hence weakly subact separable).
\end{corollary}

\begin{remark}
(1) It is not true in general that if all $M$-acts are strongly subact separable then all $M$-acts are residually finite.  Indeed, for a group $G,$ every $G$-act is strongly subact separable (Corollary \ref{all G-acts SSS}), but the $G$-act $\textbf{G}$ is residually finite if and only if $G$ is residually finite (Corollary \ref{group}(1)).\par
(2) We cannot replace `strongly subact separable' with `completely separable' in the statement of Corollary \ref{all RF implies all SSS}.  Indeed, for a group $G,$ every $G$-act is completely separable if and only if $G$ is finite (Corollary \ref{all G-acts CS}).  However, by Theorem \ref{all G-acts RF} below there exist infinite groups (such as $\mathbb{Z}$) whose acts are all residually finite. 
\end{remark}

The following result provides equivalent characterisations of the property that all finitely generated $M$-acts are strongly subact separable.

\begin{prop}
\label{all fg acts SSS}
The following are equivalent for a monoid $M$:
\begin{enumerate}
\item all finitely generated $M$-acts are strongly subact separable;
\item all cyclic $M$-acts are strongly subact separable;
\item all finitely generated $M$-acts are weakly subact separable;
\item all cyclic $M$-acts are weakly subact separable;
\item all cyclic $M$-acts satisfy the separability condition with respect to the collection of all finite subacts;
\item all cyclic $M$-acts satisfy the separability condition with respect to the collection of all trivial subacts;
\item for any cyclic $M$-act $A$ containing zeroes, any zero $0\in A$ and any $a\neq 0,$ there exist a finite $M$-act $C$ and an $M$-homomorphism $\theta : A\to C$ such that $a\theta\neq 0\theta.$
\end{enumerate}
\end{prop}

\begin{proof}
(1)\,$\Leftrightarrow$\,(2) and (3)\,$\Leftrightarrow$\,(4) are restatements of Theorem \ref{all cyclic acts} for strong subact separability and weak subact separability, respectively.  The implications (2)\,$\Rightarrow$\,(4), (4)\,$\Rightarrow$\,(5), (5)\,$\Rightarrow$\,(6) and (6)\,$\Rightarrow$\,(7) are obvious.  The proof of (7)\,$\Rightarrow$\,(1) is essentially the same as that of Proposition \ref{all acts SSS}.  The only difference is that we take $A$ to be cyclic and note that the Rees quotient $A/B$ is also cyclic.
\end{proof}

\begin{corollary}
\label{all cyclic RF implies all fg SSS}
Let $M$ be a monoid.  If all cyclic $M$-acts are residually finite, then all finitely generated $M$-acts are strongly subact separable (and hence weakly subact separable).
\end{corollary}

\begin{proof}
This follows from the equivalence of (1) and (7) in Proposition \ref{all fg acts SSS}, along with the definition of residual finiteness.
\end{proof}

We now exhibit a large class of monoids, including all groups, that have the property that all their acts are strongly subact separable.

\begin{prop}
\label{finitely many R-classes}
Let $M$ be a monoid with finitely many $\mathcal{R}$-classes.  Then all $M$-acts are strongly subact separable.
\end{prop}

\begin{proof}
Let $A$ be an $M$-act with a zero $0,$ and let $a\in A\!\setminus\!\{0\}.$  By Proposition \ref{all acts SSS}, it suffices to prove that we can separate $a$ and $0.$  Let $R_i, i\in I,$ be the $\mathcal{R}$-classes of $M.$  Note that for any $x\in A$ and $i\in I,$ either $xR_i=\{0\}$ or $0\notin xR_i.$  For each $i\in I,$ we define a map 
$$\theta_i : A\!\setminus\!\{0\}\to\{1, 2\}, x\mapsto
\begin{cases}
1&\text{ if }xR_i=\{0\}\\
2&\text{ otherwise.}
\end{cases}$$
We now define an equivalence relation $\sim$ on $A$ by
$$x\sim y\iff 
\begin{cases}
x=y=0,\;\text{or}\\
x, y\in A\!\setminus\!\{0\}\text{ and }x\theta_i=y\theta_i\text{ for each }i\in I.
\end{cases}$$
The relation $\sim$ has finite index since $I$ is finite and $|\text{Im }\theta_i|=2$ for each $i\in I.$  Clearly $\sim$ separates $a$ from $0.$ 
To complete the proof, we need to show that $\sim$ is a congruence.  So, let $x\sim y$ and $m\in M.$  Clearly we may assume that $x, y\in A\!\setminus\!\{0\}.$  Let $i\in I.$  Recalling that $\mathcal{R}$ is a left congruence on $M,$ we have that $mR_i\subseteq R_j$ for some $j\in I.$  Therefore, we have that 
\begin{align*}
(xm)\theta_i=1&\iff(xm)R_i=\{0\}\iff xR_j=\{0\}\iff x\theta_j=1\iff y\theta_j=1\\&\iff yR_j=\{0\}\iff(ym)R_i=\{0\}\iff(ym)\theta_i=1,
\end{align*}
so $(xm)\theta_i=(ym)\theta_i.$
Since $i$ was chosen arbitrarily, we have shown that $xm\sim ym,$ as required.
\end{proof}

\begin{corollary}
\label{all G-acts SSS}
Let $G$ be a group.  Then all $G$-acts are strongly subact separable.
\end{corollary}

We conclude this section with a couple of results that will be useful later in the paper.

\begin{prop}
\label{submonoid}
Let $M$ be a monoid with a submonoid $N$ such that $M\!\setminus\!N$ is an ideal.  Let $\mathcal{C}$ be any of the four separability conditions.  If all (finitely generated) $M$-acts satisfy $\mathcal{C},$ then all (finitely generated) $N$-acts satisfy $\mathcal{C}.$
\end{prop}

\begin{proof}
Let $A$ be an $N$-act, let $X\subseteq A$ be of the type associated with $\mathcal{C},$ and let $a\in A\!\setminus\!X.$  (For residual finiteness, we assume that $X$ is a finite subset).  Let $0$ be an element not in $A,$ and let $A^{\prime}=A\cup\{0\}.$  We make $A^{\prime}$ into an $M$-act by defining
$$a^{\prime}\cdot m=
\begin{cases} 
a^{\prime}m&\text{if }a^{\prime}\in A\text{ and }m\in N\\ 
0&\text{otherwise.}
\end{cases}$$
Notice that any set that generates the $N$-act $A$ also generates the $M$-act $A^{\prime}.$  In particular, if $A$ is finitely generated then so is $A^{\prime}.$
Clearly $X\cup\{0\}$ is of type associated with $\mathcal{C}$ in $A^{\prime}.$
Since $A^{\prime}$ satisfies $\mathcal{C},$ there exist a finite $M$-act $C$ and an $M$-homomorphism $\theta : A^{\prime}\to C$ such that $a\notin A^{\prime}\!\setminus\!(X\cup\{0\}).$  The set $C$ is an $N$-act via the restriction of the action of $M$ to $N.$  The restriction $\theta|_A : A\to C$ is easily seen to be an $N$-homomorphism, and clearly $\theta|_A$ separates $a$ from $X.$  Thus $A$ satisfies $\mathcal{C}.$
\end{proof}

A submonoid $N$ of a monoid $M$ is said to be a {\em retract} of $M$ if there exists a homomorphism $\phi : M\to N$ such that $n\phi=n$ for all $n\in N.$

\begin{prop}
\label{retract}
Let $M$ be a monoid and let $N$ be a retract of $M.$  Let $\mathcal{C}$ be any of the four separability conditions.  If all (finitely generated) $M$-acts satisfy $\mathcal{C},$ then all (finitely generated) $N$-acts satisfy $\mathcal{C}.$
\end{prop}

\begin{proof}
Let $A$ be an $N$-act, let $X\subseteq A$ by of the type associated with $\mathcal{C},$ and let $a\in A\!\setminus\!X.$  Since $N$ is a retract of $M,$ there exists a homomorphism $\phi : M\to N$ such that $n\phi=n$ for all $n\in N.$  The set $A$ is clearly an $M$-act under the action given by $x\cdot m=x(m\phi).$  It is clear that any set that generates $A$ as an $M$-act also generates $A$ as an $N$-act.  In particular, if $A$ is finitely generated as an $M$-act then it is finitely generated as an $N$-act.  Since $A$ satisfies $\mathcal{C}$ as an $M$-act, there exist a finite $M$-act $C$ and an $M$-homomorphism $\theta : A\to C$ such that $a\theta\notin X\theta.$  The set $C$ is an $N$-act via the restriction of the action of $M$ to $N,$ and it is easily verified that this turns $\theta$ into an $N$-homomorphism.  Since $a\theta\notin X\theta,$ we have shown that $A$ satisfies $\mathcal{C}$ as an $N$-act.
\end{proof}

\section{\large{Some Special Classes of Monoids}\nopunct}
\label{sec:regular}

In this section we study the finiteness conditions introduced in the previous section for certain special classes of monoids.  Specifically, we consider commutative monoids, groups, completely simple semigroups with identity adjoined, and Clifford monoids.

\vspace{0.5em}
\subsection{Commutative monoids\nopunct}
~\par\vspace{0.5em}

It is well known that every finitely generated commutative monoid is residually finite \cite[Theorem 3]{Lallement}, and clearly every quotient of a finitely generated commutative monoid is itself finitely generated and commutative, so from Corollaries \ref{fg acts, quotients} and \ref{all cyclic RF implies all fg SSS} we obtain:

\begin{thm}
\label{fg comm}
Let $M$ be a finitely generated commutative monoid.  Then all finitely generated $M$-acts are residually finite and strongly subact separable.
\end{thm}

\begin{remark}
Theorem \ref{fg comm} does not hold if we remove the condition that $M$ is finitely generated.  For example, if $M$ is a free commutative monoid with an infinite basis, then there certainly exist quotients of $M$ that are not residually finite, and hence there exist finitely generated $M$-acts that are not residually finite by Corollary \ref{fg acts, quotients}.
\end{remark}

For a finitely generated commutative monoid, being completely separable is equivalent to every $\mathcal{H}$-class being finite \cite[Theorem 4.3]{Miller}, so we obtain the following from Corollary \ref{all fg acts CS}.

\begin{prop}
\label{fg comm, CS}
Let $M$ be a finitely generated commutative monoid.  Then the following are equivalent:
\begin{enumerate}  
\item all finitely generated $M$-acts are completely separable;
\item for every congruence $\rho$ on $M,$ every $\mathcal{H}$-class of the quotient $M/\rho$ is finite.
\end{enumerate}
\end{prop}

\begin{corollary}
\label{N, CS}
Let $M$ denote the free monogenic monoid.  Then all finitely generated $M$-acts are completely separable.
\end{corollary}

\begin{proof}
It is clear that every $\mathcal{H}$-class of $M$ is a singleton.  For any non-trivial congruence $\rho$ on $M,$ the quotient $M/\rho$ is finite, so certainly every $\mathcal{H}$-class of $M/\rho$ is finite.  Therefore, by Proposition \ref{fg comm, CS}, all finitely generated $M$-acts are completely separable.
\end{proof}

It is not the case that every act over the free monogenic monoid is completely separable (or even residually finite), as the following example demonstates.

\begin{ex}
Let $M=\langle x\rangle\cong\mathbb{N}_0,$ and let $A=\{a_i : i\in\mathbb{N}_0\}\cup\{0\}.$  We define an action of $M$ on $A$ by $0x^j=0$ and
$$a_ix^j=\begin{cases}
a_{i-j}&\text{if }i\geq j,\\
0&\text{otherwise.}
\end{cases}$$
We show that $A$ is indeed an $M$-act under this action.  Clearly $a1=a$ for all $a\in A$ and $(0x^j)x^k=0x^{j+k}$ for all $j, k\in\mathbb{N}_0.$  Now let $i, j, k\in\mathbb{N}_0.$  We need to show that $(a_ix^j)x^k=a_ix^{j+k}.$  Suppose first that $i\geq j+k.$  Then certainly $i\geq j$ and $i-j\geq k.$  Therefore, we have that
$$(a_ix^j)x^k=a_{i-j}x^k=a_{i-j-k}=a_{i-(j+k)}=a_ix^{j+k}.$$
Now suppose that $i<j+k.$  Then $i-j<k.$  Now $a_ix^j$ is either $0$ or $a_{i-j}.$  In either case we have that $(a_ix^j)x^k=0=a_ix^{j+k},$ as desired.\par
We now claim that $0$ and $a_0$ cannot be separated in a finite index congruence on $A,$ so $A$ is not residually finite.  Indeed, let $\sim$ be a finite index congruence on $A.$  Then there exist $i, j\in\mathbb{N}_0$ with $i<j$ such that $a_i\sim a_j.$  It follows that $$0=a_ix^j\sim a_jx^j=a_0,$$ as required.
\end{ex}

\begin{remark}
Let $M$ be a free commutative monoid on at least two generators.  Then the group $\mathbb{Z}$ is a quotient of $M,$ and of course $\mathbb{Z}$ has an infinite $\mathcal{H}$-class, namely itself.  Therefore, there exists a cyclic $M$-act that is not completely separable by Proposition \ref{fg comm, CS} and Theorem \ref{all cyclic acts}.
\end{remark}

A semigroup $S$ is said to be {\em null} if it has a zero $0$ and $st=0$ for all $s, t\in S.$  Kozhukhov showed in the proof of \cite[Lemma 1]{Kozhukhov2} that for any infinite null semigroup $S$ there exists an infinite subdirectly indecomposable $S$-act.  It can be easily shown that such an act is not residually finite.  We now present Kozhukhov's construction, turn it into an $S^1$-act, and show that it is neither residually finite nor weakly subact separable.

\begin{ex}[{\cite[Proof of Lemma 1]{Kozhukhov2}}]
Let $S=\{s_i : i\in I\}\cup\{z\}$ with $I$ infinite, and define $st=z$ for all $s, t\in S.$  Thus $S$ is a null semigroup with zero $z.$  Let $A=\{a_i : i\in I\}\cup\{b, 0\}.$  Define an action of $S$ on $A$ by
$$a_is_j=\begin{cases}
b&\text{if }i=j\\
0&\text{if }i\neq j,
\end{cases}
\;\text{ and }\;bs_j=0s_j=az=0\;\; (a\in A, i, j\in I).$$
We have that $a(st)=az=0=(as)t$ for all $a\in A$ and $s, t\in S.$  Thus, setting $a1=a$ for all $a\in A,$ the set $A$ is an $S^1$-act.  We claim that we cannnot separate $b$ from $\{0\}$ in a finite index congruence, so $A$ is neither residually finite nor weakly subact separable.  Indeed, if $\sim$ is a finite index congruence on $A,$ then there exist $i, j\in I$ with $i\neq j$ such that $a_i\sim a_j,$ and hence $b=a_is_i\sim a_js_i=0,$ as required.
\end{ex}

We now show that every finitely generated act over a null semigroup with identity adjoined is completely separable.

\begin{prop}
\label{null}
Let $S$ be a null semigroup and let $M=S^1.$  Then all finitely generated $M$-acts are completely separable.
\end{prop}

\begin{proof}
Let $N$ be any quotient of $M.$  It is easy to see that $N=T^1$ for some null semigroup $T.$  
Recalling Definition \ref{CS defn}, for any $a\in N\!\setminus\!\{0\}$ we have 
$$\{[a, b] : b\in N\}=\{[a, 1], [a, a], \emptyset\}=\bigl\{\{a\}, \{1\}, \emptyset\bigr\},$$
and $$\{[0, b] : b\in N\}=\{[0, 1], [0, 0], T\}=\bigl\{\{0\}, N, T\bigr\}.$$
It follows from Corollary \ref{CS monoid} that $N$ is completely separable.  Hence, by Corollary \ref{all fg acts CS}, all finitely generated $M$-acts are completely separable.
\end{proof}

\begin{prob}
Does there exist an infinite commutative monoid whose acts are all completely separable?
\end{prob}

\vspace{0.5em}
\subsection{Groups\nopunct}
~\par\vspace{0.5em}

We have seen that, for a group $G,$ all $G$-acts are strongly subact separable, and all $G$-acts are completely separable if and only if $G$ is finite.
Kozhukhov proved that all $G$-acts are residually finite if and only if every subgroup of $G$ is the intersection of a family of finite index subgroups \cite[Theorem 7]{Kozhukhov3}.  We now reprove his result and also find some new characterisations.

\begin{thm}
\label{all G-acts RF}
The following are equivalent for a group $G$:
\begin{enumerate}
\item all $G$-acts are residually finite;
\item all finitely generated $G$-acts are residually finite;
\item all cyclic $G$-acts are residually finite;
\item every subgroup of $G$ is the intersection of a family of finite index subgroups;
\item $G$ is strongly subgroup separable.
\end{enumerate}
\end{thm}

\begin{proof}
(1)\,$\Rightarrow$\,(2) and (2)\,$\Rightarrow$\,(3) are obvious.\par
(3)\,$\Rightarrow$\,(1).  Since every $G$-act is the disjoint union of cyclic $G$-acts, this follows from Proposition \ref{disjoint union}.\par
(3)\,$\Leftrightarrow$\,(4).  Since (finite index) right congruences on $G$ are in one-to-one correspondence with (finite index) subgroups of $G,$ this follows from Corollary \ref{all fg acts RF}.\par
(3)\,$\Rightarrow$\,(5).  Let $H$ be a subgroup of $G$ and let $x\in G\!\setminus\!H.$  The set $G/H=\{Hg : g\in G\}$ of right cosets of $H$ is a $G$-act via $Hg\cdot g^{\prime}=H(gg^{\prime}).$  Clearly $G/H$ is cyclic, so it is residually finite by assumption.  We have that $Hx\neq H$ since $x\notin H.$  Therefore, there exist a finite $G$-act $A$ and a $G$-homomorphism $\theta : G/H\to A$ such that $(Hx)\theta\neq H\theta.$  Now, we have a group homomorphism $\psi : G\to S_A,$ where $S_A$ denotes the symmetric group on $A,$ given by $a(g\psi)=ag$ for all $a\in A.$  For each $h\in H$ we have
$$(H\theta)(h\psi)=(H\theta)h=(Hh)\theta=H\theta\neq(Hx)\theta=(H\theta)x=(H\theta)(x\psi),$$
so $x\psi\neq h\psi.$  Thus $\psi$ separates $x$ from $H,$ and hence $G$ is strongly subgroup separable.\par
(5)\,$\Rightarrow$\,(3).  Let $A$ be a cyclic $G$-act.  Then $A\cong G/H=\{Hg : g\in G\}$ for some subgroup $H$ of $G.$  We need to show that $G/H$ is residually finite.  So, let $x, y\in G$ with $Hx\neq Hy.$  Then $xy^{-1}\notin H.$  Since $G$ is strongly subgroup separable, there exist a finite group $K$ and a group homomorphism $\theta : G\to K$ such that $(xy^{-1})\theta\notin H\theta.$  Let $L$ denote the subgroup $H\theta$ of $K.$  The set $K/L=\{Lk : k\in K\}$ of right cosets of $L$ is a $G$-act via $Lk\cdot g=L\bigl(k(g\theta)\bigr).$  Certainly $K/L$ is finite as $K$ is finite.  We have a well-defined map
$$\phi : G/H\to K/L, Hg\mapsto L\cdot g.$$
For any $g, g^{\prime}\in G,$ we have
$$\bigl((Hg)g^{\prime})\phi=\bigl(H(gg^{\prime})\bigr)\phi=L\cdot gg^{\prime}=(L\cdot g)\cdot g^{\prime}=(Hg)\phi\cdot g^{\prime},$$
so $\phi$ is a $G$-homomorphism.  Since $(xy^{-1})\theta\notin H\theta,$ we deduce that
$$(Hx)\phi=L\cdot x=(H\theta)(x\theta)=(Hx)\theta\neq(Hy)\theta=(H\theta)(y\theta)=L\cdot y=(Hy)\phi.$$
Thus $\phi$ separates $Hx$ and $Hy,$ as desired.
\end{proof}

Mal'cev proved in \cite{Malcev} that all polycyclic-by-finite groups (which include finitely generated nilpotent groups and hence finitely generated abelian groups) are strongly subgroup separable, so we deduce:

\begin{corollary}
If $G$ is a polycyclic-by-finite group, then all $G$-acts are residually finite.
\end{corollary}

Smirnov proved in \cite{Smirnov} that a nilpotent group $G$ is strongly subgroup separable if and only if every quotient of $G$ by a normal subgroup is residually finite (see \cite[Theorem 4.1]{Robinson} for a short proof).  This result and Theorem \ref{all G-acts RF} together yield:

\begin{corollary}
\label{nilpotent}
Let $G$ be a nilpotent group.  Then the following are equivalent:
\begin{enumerate}
\item all $G$-acts are residually finite;
\item for every normal subgroup $N$ of $G,$ the quotient $G/N$ is residually finite.
\end{enumerate}
\end{corollary}

\begin{remark}
\label{solvable}
Corollary \ref{nilpotent} does not hold for solvable groups.
Indeed, by \cite[Remark 3]{Robinson}, the solvable group $G$ defined by the presentation $\langle a, b\,|\,b^{-1}ab=a^2\rangle$ satisfies the condition that all its quotients are residually finite, but $G$ is not strongly subgroup separable.  
\end{remark}

\begin{remark}
It follows from Theorem \ref{all G-acts RF} and Remark \ref{solvable} that the converse of Corollary \ref{all quotients RF} (for residual finiteness) does not hold.
\end{remark}

\vspace{0.5em}
\subsection{Completely simple semigroups with identity adjoined\nopunct}
~\par\vspace{0.5em}

Recall that a semigroup $S$ is {\em completely simple} if it is simple (that is, it contains no proper ideals) and possesses minimal left and right ideals.
Completely simple semigroups have the following well-known representation, due to Rees.  Let $G$ be a group, let $I$ and $J$ be two non-empty index sets, and let $P=(p_{ji})$ be a $J\times I$ matrix with entries from $G.$
The set $I\times G\times J$ with multiplication given by $$(i, g, j)(k, h, l)=(i, gp_{jk}h, l)$$
is a semigroup, called the {\em Rees matrix semigroup over $G$ with respect to $P$} and denoted by $\mathcal{M}(G; I, J; P).$
Rees' Theorem states that a semigroup $S$ is completely simple if and only if it is isomorphic to a Rees matrix semigroup $\mathcal{M}(G; I, J; P)$ (see \cite[Theorem 3.3.1]{Howie}).  In fact, by \cite[Theorem 3.4.2]{Howie}, a semigroup is completely simple if and only if it is isomorphic to a Rees matrix semigroup $\mathcal{M}(G; I, J; P)$ where $P$ is {\em normalised} (that is, there exist a row and a column of $P$ in which all the entries are the identity of $G$).\par
Golubov found necessary and sufficient conditions for a completely simple semigroup to be residually finite in \cite{Golubov2}.  Before stating his result, we first need to make some definitions.\par
Consider a completely simple semigroup $\mathcal{M}(G; I, J; P).$  Define an equivalence relation $\sim_I$ on $I$ by $i\sim_Ik$ if and only if there exists $g\in G$ such that $p_{ji}=p_{jk}g$ for all $j\in J.$  Let $r_I$ denote the index of $\sim_I$.  Similarly, define an equivalence relation $\sim_J$ on $J$ by $j\sim_Jl$ if and only if there exists $g\in G$ such that $p_{ji}=gp_{li}$ for all $i\in I,$ and let $r_J$ denote the index of $\sim_J$.  The {\em rank} of $P$ is defined to be $\max(r_I, r_J).$  Given a normal subgroup $N$ of $G,$ let $P/N$ denote the $J\times I$ matrix $(p_{ji}N)_{j\in J, i\in I}.$

\begin{thm}[{\cite[Theorem 3]{Golubov2}}]
\label{Golubov}
A completely simple semigroup $\mathcal{M}(G; I, J; P)$ is residually finite if and only if $G$ is residually finite and $P/N$ has finite rank for every normal subgroup $N$ of $G$ of finite index.
\end{thm}

In the remainder of this subsection, unless otherwise stated, $M$ will denote a monoid $S^1$ where $S$ is a completely simple semigroup.  We use Theorem \ref{Golubov} to show that the cyclic $M$-act $\textbf{M}$ may not be residually finite, even in the case that its subgroups are finite. 

\begin{ex}
Let $G=\{e, g\}$ be the cyclic group of order 2.  Let $I$ be any infinite set, fix $i_0\in I,$ and let $P=(p_{ij})$ be the $I\times I$ matrix where $p_{i_0i_0}=e,$ $p_{ii}=g$ for all $i\in I\!\setminus\!\{i_0\},$ and $p_{jk}=e$ for all $j, k\in I, j\neq k.$  Let $S=\mathcal{M}(G; I, I; P).$  It is clear that $P$ has infinite rank, so $S$ is not residually finite by Theorem \ref{Golubov}.  Thus $M=S^1$ is not residually finite, and hence $\textbf{M}$ is not residually finite by Corollary \ref{monoid separability}(1).\par
For completeness, we also directly prove that $S$ is not residually finite.  We claim that the elements $(i_0, e, i_0)$ and $(i_0, g, i_0)$ cannot be separated in a finite index congruence.  Indeed, let $\sim$ be a finite index congruence on $S.$  Then there exist $j, k\in I$ with $j\neq k$ such that $(i_0, e, j)\sim(i_0, e, k).$  It follows that $$(i_0, e, i_0)=(i_0, e, j)(k, e, i_0)\sim(i_0, e, k)(k, e, i_0)=(i_0, g, i_0),$$ as required.
\end{ex}

The following lemma will be useful in determining necessary conditions for all $M$-acts to be residually finite or all $M$-acts to be strongly subact separable.

\begin{lemma}
\label{left zero}
Let $L$ be an infinite left zero semigroup.  Then there exists an $L^1$-act that is neither residually finite nor weakly subact separable.
\end{lemma}

\begin{proof}
Let $A=\{a_x : x\in L\}\cup\{b, c\}.$  Define $a1=a$ for all $a\in A,$ and
$$a_xy=
\begin{cases}
b&\text{if }x=y\\
c&\text{if }x\neq y,
\end{cases}$$
$$by=b\,\text{ and }\,cy=c$$
for all $x, y\in L.$  It is easily verified that $A$ is an $L^1$-act with the above action.  We claim that we cannot separate $b$ from the trivial subact $\{c\}$ in a finite index congruence, so $A$ is neither residually finite nor weakly subact separable.  Indeed, if $\sim$ is a finite index congruence on $A,$ then there exist $x, y\in L$ with $x\neq y$ such that $a_x\sim a_y.$  It follows that 
$$b=a_xx\sim a_yx=c,$$
as required.
\end{proof}

We now provide some necessary conditions for all $M$-acts to be residually finite.

\begin{prop}
\label{prop:CS,RF}
Let $S=\mathcal{M}(G; I, J; P)$ and let $M=S^1.$  If all $M$-acts are residually finite, then $I$ is finite and $G$ is strongly subgroup separable.
\end{prop}

\begin{proof}
We may assume that $P$ is normalised, so that there exist $i_0\in I$ and $j_0\in J$ such that $p_{ji_0}=p_{j_0i}=e$ for all $i\in I$ and $j\in J,$ where $e$ denotes the identity of $G.$
It is easy to see that $L=I\times\{1_G\}\times\{j_0\}$ is a left zero semigroup and that $L^1$ is a retract of $M$ via the homomorphism $\theta : M\to L^1$ given by $1\theta=1$ and $(i, g, j)\theta=(i, 1_G, j_0).$  Therefore, by Proposition \ref{retract}, all $L^1$-acts are residually finite.  It follows from Lemma \ref{left zero} that $L$ is finite, and hence $I$ is finite.\par
We now prove that every cyclic $G$-act is residually finite, from which it follows that $G$ is strongly subgroup separable by Theorem \ref{all G-acts RF}.  Every cyclic $G$-act is isomorphic to $G/H=\{Hg : g\in G\}$ for some subgroup $H$ of $G.$  So, let $H$ be a subgroup of $G,$ and let $x, y\in G$ be such that $Hx\neq Hy.$  We set $A=G/H\times J,$ and define an action of $M$ on $A$ by $(Hg, j)(i, g^{\prime}, k)=(H(gp_{ji}g^{\prime}), k)$ and $(Hg, j)1=(Hg, j).$  It is easily verified that $A$ is an $M$-act with the above action.  Since $A$ is residually finite, there exist a finite $M$-act $C$ and an $M$-homomorphism $\theta : A\to C$ such that $(Hx, j_0)\theta\neq(Hy, j_0)\theta.$  It is easy to see that $C$ is a $G$-act via $c\ast g=c(i_0, g, j_0).$  We now define a map 
$$\phi : G/H\to C, Hg\mapsto(Hg, j_0)\theta.$$
Using the fact that $\theta$ is an $M$-homomorphism, for any $g, g^{\prime}\in G$ we have
\begin{align*}
\bigl((Hg)g^{\prime}\bigr)\phi&=\bigl(H(gg^{\prime}))\phi=\bigl(H(gg^{\prime}), j_0\bigr)\theta=\bigl((Hg, j_0)(i_0, g^{\prime}, j_0)\bigr)\theta=\bigl((Hg, j_0)\theta\bigr)(i_0, g^{\prime}, j_0)\\&=\bigl((Hg)\phi\bigr)\ast g^{\prime},
\end{align*}
so $\phi$ is a $G$-homomorphism.  Moreover, we have that $(Hx)\phi\neq(Hy)\phi.$  Thus $G/H$ is residually finite, as desired.
\end{proof}

The converse of Proposition \ref{prop:CS,RF} does not hold.  In fact, the conditions that $I$ is finite and $G$ is strongly subgroup separable need not even imply that all {\em cyclic} $M$-acts are residually finite.

\begin{ex}
Let $G$ be any infinite group and let $e$ denote its identity.  Let $I$ and $J$ be any non-empty index sets with $|J|\geq 2.$  Fix distinct elements $j, k\in J,$ and let $P$ be any $J\times I$ matrix over $G$ such that $p_{ji}=p_{ki}=e$ for all $i\in I.$  Let $S=\mathcal{M}(G; I, J; P)$ and let $M=S^1.$  Now fix $i_0\in I$ and let
$$\rho=\bigl\{\bigl((i_0, g, j), (i_0, g, k)\bigr) : g\in G\!\setminus\!\{e\}\bigr\}\cup\{(m, m) : m\in M\}.$$
It is straightforward to check that $\rho$ is a right congruence on $M.$  Let $A$ denote the cyclic $M$-act $\textbf{M}/\rho.$  We claim that we cannot separate the distinct elements $a=[(i_0, e, j)]_{\rho}$ and $b=[(i_0, e, k)]_{\rho}$ in a finite index congruence on $A,$ and hence $A$ is not residually finite.  Indeed, let $\sim$ be a finite index congruence on $A.$  Then there exist $g, h\in G$ with $g\neq h$ such that $[(i_0, g, k)]_{\rho}\sim[(i_0, h, k)]_{\rho}.$  It follows that
$$b=[(i_0, g, k)]_{\rho}(i_0, g^{-1}, k)\sim[(i_0, h, k)]_{\rho}(i_0, g^{-1}, k)=[(i_0, hg^{-1}, k)]_{\rho}=[(i_0, hg^{-1}, j)]_{\rho},$$
and thus $$a=b(i_0, e, j)\sim[(i_0, hg^{-1}, j)]_{\rho}(i_0, e, j)=[(i_0, hg^{-1}, j)]_{\rho}\sim b,$$
as required.
\end{ex}

\begin{prob}
Let $S=\mathcal{M}(G; I, J; P)$ and let $M=S^1.$  Under what conditions are all (finitely generated) $M$-acts residually finite?
\end{prob}

The following result determines when $M$ satisfies the condition that all $M$-acts are strongly subact separable (or, equivalently, all $M$-acts are weakly subact separable).

\begin{thm}
\label{CS, RF}
Let $S=\mathcal{M}(G; I, J; P)$ and let $M=S^1.$  Then the following are equivalent:
\begin{enumerate}
\item all $M$-acts are strongly subact separable;
\item all $M$-acts are weakly subact separable;
\item $I$ is finite.
\end{enumerate}
\end{thm}

\begin{proof}
(1) and (2) are equivalent by Proposition \ref{all acts SSS}.\par
(2)\,$\Rightarrow$\,(3).  It was shown in Proposition \ref{prop:CS,RF} that the set $L^1,$ where $L=I\times\{1_G\}\times\{1\},$ is a retract of $M.$  Therefore, all $L^1$-acts are weakly subact separable by Proposition \ref{retract}.  It follows from Lemma \ref{left zero} that $I$ is finite, since otherwise there would exist a weakly subact separable $L^1$-act.\par
(3)\,$\Rightarrow$\,(1).  This follows from Proposition \ref{finitely many R-classes}, since $M$ has finitely many $\mathcal{R}$-classes.
\end{proof}

\begin{thm}
Let $S=\mathcal{M}(G; I, J; P)$ and let $M=S^1.$  Then all finitely generated $M$-acts are strongly subact separable.
\end{thm}

\begin{proof}
Let $A$ be a cyclic $M$-act with a zero $0,$ and consider an element $a\in A\!\setminus\!\{0\}.$  By Proposition \ref{all fg acts SSS}, it suffices to prove that we can separate $a$ and $0.$  Now, by Proposition \ref{cyclic act, right congruence}, we may assume that $A=\textbf{M}/\rho$ for some right congruence $\rho$ on $M.$  Clearly $[1]_{\rho}\neq 0,$ since this would imply that $a=0.$  We shall prove that $[1]_{\rho}=\{1\}$ and that $A\!\setminus\!\{[1]_{\rho}, 0\}$ is a subact of $A.$  It is then clear that the equivalence relation on $A$ with classes $\{[1]_{\rho}\}, \{0\}$ and $A\!\setminus\!\{[1]_{\rho}, 0\}$ is a congruence on $A$ that separates $a$ and $0,$ as desired.\par  
Consider $b=[(i, g, j)]_{\rho}$ with $b\neq 0,$ and suppose that $bs=0$ for some $s=(k, h, l)\in S.$  Letting $t=(i, p_{li}^{-1}h^{-1}p_{jk}^{-1}, j),$ we have that $(i, g, j)=(i, g, j)st.$  But then $$b=(bs)t=0t=0,$$ which is a contradiction, so $bs\neq 0.$\par
Now suppose for a contradiction that $1\,\rho\,(i, g, j)$ for some $(i, g, j)\in S.$  Letting $b=[(i, g, j)]_{\rho}$ and choosing $s\in S$ such that $0=[s]_{\rho},$ we obtain $$bs=[(i, g, j)]_{\rho}s=[1]_{\rho}s=[s]_{\rho}=0.$$  But we have already established that $bs\neq 0,$ so we have a contradiction and $[1]_{\rho}=\{1\}.$\par
Now, for any $c\in A\!\setminus\!\{[1]_{\rho}, 0\}$ and $s\in S,$ it is clear that $cs\neq[1]_{\rho},$ and the argument above proves that $cs\neq 0,$ so we conclude that $A\!\setminus\!\{[1]_{\rho}, 0\}$ is a subact of $A.$  This completes the proof.
\end{proof}

In what follows we shall determine when all $M$-acts are completely separable and when all finitely generated $M$-acts are completely separable.  In order to do so, we first provide the following technical result, which characterises certain sets of the form given in Definition \ref{CS defn}.

\begin{lemma}
\label{CS lemma}
Let $S=\mathcal{M}(G; I, J; P)$ where $P$ is normalised; in particular, there exists $i_0\in I$ such that $p_{ji_0}=e$ for all $j\in J,$ where $e$ denotes the identity of $G.$  Let $M$ denote the monoid $S^1,$ let $A$ be an $M$-act, and let $a\in A.$
For any $b\in A\!\setminus\!\{a\}$ with $a\leq_A b,$ we have $[a, b]_M=U_b\times J^{\prime},$ where
$$U_b=\{(i, g)\in I\times G : a=b(i, g, j)\emph{ for some }j\in J\}, \;J^{\prime}=\{j\in J : a=a(i_0, e, j)\}.$$
\end{lemma}

\begin{proof}
Let $b\in A\!\setminus\!\{a\}$ be such that $a\leq_A b.$  First let $(i, g, j)\in[a, b]_M.$  Then $a=b(i, g, j),$ so $(i, g)\in U_b.$  Moreover, we have that $$a=b(i, g, j)=b(i, g, j)(i_0, e, j)=a(i_0, e, j),$$ so $j\in J^{\prime}.$  Thus $(i, g, j)\in U_b\times J^{\prime},$ and hence $[a, b]_M\subseteq U_b\times J^{\prime}.$\par
For the reverse containment, let $(i, g, j)\in U_b\times J^{\prime}.$  Then there exists $k\in J$ such that $a=b(i, g, k),$ and $a=a(i_0, e, j).$  It follows that $$a=b(i, g, k)(i_0, e, j)=b(i, g, j),$$ so $(i, g, j)\in[a, b]_M.$  Thus $U_b\times J^{\prime}\subseteq[a, b]_M,$ and hence $[a, b]_M=U_b\times J^{\prime}.$
\end{proof}

\begin{thm}
\label{CS, CS}
Let $S=\mathcal{M}(G; I, J; P)$ and let $M=S^1.$  Then all $M$-acts are completely separable if and only if both $G$ and $I$ are finite.
\end{thm}

\begin{proof}
($\Rightarrow$) Since the $M$-act $\textbf{M}$ is completely separable, the monoid $M$ is completely separable by Corollary \ref{monoid separability}(1).  Since $G$ is (isomorphic to) a maximal subgroup of $M,$ it is completely separable by Lemma \ref{subalgebra} and hence finite by \cite[Lemma 2.4]{Miller}.
Since all $M$-acts are completely separable, certainly all $M$-acts are residually finite, so $I$ is finite by Proposition \ref{prop:CS,RF}.\par
($\Leftarrow$) We may assume that $P$ is normalised.  Let $A$ be an $M$-act.  We shall prove that $A$ is completely separable by showing that is satisfies the condition of Theorem \ref{CS condition}.  So, let $a\in A.$  We write $\leq$ for the preorder $\leq_A\!\!.$  For any $b\in A$ with $a\not\leq b,$ we have $[a, b]_M=\emptyset.$  Therefore, it suffices to prove that the set $$H=\{[a, b]_M : b\in A\!\setminus\!\{a\}, a\leq b\}$$ is finite.  It follows from Lemma \ref{CS lemma} that $H$ is finite if and only if there are only finitely many sets of the form $U_b$ $(b\in A\!\setminus\!\{a\}, a\leq b),$ where $U_b$ is as defined in Lemma \ref{CS lemma}.  It is easy to see that this is the case, since $I\times G$ is finite and each $U_b$ is contained in $I\times G.$
\end{proof}

\begin{thm}
\label{CS, CS, fg}
Let $S=\mathcal{M}(G; I, J; P)$ and let $M=S^1.$  Then all finitely generated $M$-acts are completely separable if and only if $G$ is finite and $P$ has finite rank.
\end{thm}

\begin{proof}
($\Rightarrow$) By the same argument as the one in the proof of Theorem \ref{CS, CS}, we have that $M$ is completely separable and $G$ is finite.  Then $M$ is certainly residually finite, and hence $S$ is residually finite.  It follows from Theorem \ref{Golubov} that the matrix $P$ has finite rank.\par
($\Leftarrow$) By Theorem \ref{all cyclic acts}, it suffices to prove that all cyclic $M$-acts are completely separable.  So, let $A$ be a cyclic $M$-act and let $a\in A.$  By Proposition \ref{cyclic act, right congruence}, we may assume that $A=\textbf{M}/\rho$ for some right congruence $\rho$ on $M.$  We write $\leq$ for the preorder $\leq_A\!\!.$  We may assume that $P$ is normalised; in particular, there exists $i_0\in I$ such that $p_{ji_0}=e$ for all $j\in J,$ where $e$ denotes the identity of $G.$  By the same reasoning as that in the proof of Theorem \ref{CS, CS}, it is enough to prove that there are only finitely many sets of the form $U_b$ $(b\in A\!\setminus\!\{a\}, a\leq b),$ where $U_b$ is as defined in Lemma \ref{CS lemma}.\par
If $a=[1]_{\rho}$ and $[1]_{\rho}=\{1\},$ then there are no $b\in A\!\setminus\!\{a\}$ with $a\leq b.$  We may assume then that $a=[(i_a, g_a, j_a)]_{\rho}$ for some $(i_a, g_a, j_a)\in S.$  For each $b\in A\!\setminus\!\{a, [1]_{\rho}\},$ choose a representative $(i_b, g_b, j_b)\in S$ such that $b=[(i_b, g_b, j_b)]_{\rho}.$
Also, let
$$Z_b=\{h\in G : a=[(i_b, h, j_a)]_{\rho}\}.$$
Note that if $j\sim_Jl,$ then $p_{ji}=p_{li}$ for all $i\in I.$  (Indeed, if $p_{ji}=gp_{li}$ for all $i\in I,$ then $e=p_{ji_0}=gp_{li_0}=ge=g.$)\par
We claim that, for any $b, c\in A\!\setminus\!\{a\}$ with $a\leq b, a\leq c,$ if $g_b=g_c,$ $Z_b=Z_c$ and $j_b\sim_Jj_c,$ then $U_b=U_c.$  Since $G$ is finite and $\sim_j$ has finite index, it then follows that there are only finitely many sets of the form $U_b,$ as desired.\par
To prove the claim, first let $(k, h)\in U_b.$  Then there exists $l\in J$ such that $a=b(k, h, l).$  It follows that $(i_a, g_a, j_a)\,\rho\,(i_b, g_bp_{j_bk}h, l).$  Post-multiplying by $(i_0, e, j_a),$ we obtain $(i_a, g_a, j_a)\,\rho\,(i_b, g_bp_{j_bk}h, j_a),$ so $g_bp_{j_bk}h\in Z_b.$  From the assumptions we deduce that $g_cp_{j_ck}h=g_bp_{j_bk}h\in Z_c.$  Thus $(i_a, g_a, j_a)\,\rho\,(i_c, g_cp_{j_ck}h, j_a).$  It follows that $a=c(k, h, j_a),$ and hence $(k, h)\in U_c.$  Thus $U_b\subseteq U_c.$  A dual argument proves that $U_c\subseteq U_b,$ so $U_b=U_c,$ as required.
\end{proof}

\begin{corollary}
Let $S=\mathcal{M}(G; I, J; P)$ and let $M=S^1.$  If at least one of the index sets $I$ and $J$ is finite, then all finitely generated $M$-acts are completely separable if and only if $G$ is finite.
\end{corollary}

\begin{proof}
If all finitely generated $M$-acts are completely separable, then $G$ is finite by Theorem \ref{CS, CS, fg}.  Conversely, if $G$ is finite, then since at least one of the index sets $I$ and $J$ is finite, the rank of the matrix $P$ is finite by \cite[Lemma 7]{Golubov2}.  Hence, by Theorem \ref{CS, CS, fg}, all finitely generated $M$-acts are completely separable.
\end{proof}

A {\em rectangular band} is a direct product of a left zero semigroup and a right zero semigroup; equivalently, it is a completely simple semigroup whose subgroups are all trivial.

\begin{corollary}
Let $S=I\times J$ be a rectangular band and let $M=S^1.$  Then the following are equivalent:
\begin{enumerate}
\item all $M$-acts are residually finite;
\item all $M$-acts are weakly subact separable;
\item all $M$-acts are strongly subact separable;
\item all $M$-acts are completely separable;
\item $I$ is finite.
\end{enumerate}
\end{corollary}

\begin{proof}
(1) implies (5) by Proposition \ref{prop:CS,RF}, (5) implies (4) by Theorem \ref{CS, CS}, and it is obvious that (4) implies (1).  Finally, (2), (3) and (5) are equivalent by Theorem \ref{CS, RF}.
\end{proof}

From Theorem \ref{CS, CS, fg} we immediately deduce:

\begin{corollary}
Let $S$ be a rectangular band and let $M=S^1.$  Then all finitely generated $M$-acts are completely separable.
\end{corollary}

\vspace{0.5em}
\subsection{Clifford monoids\nopunct}
~\par\vspace{0.5em}

A {\em Clifford monoid} is an inverse monoid whose idempotents are central (i.e.\ they commute with every element).  Clifford monoids can be nicely described in terms of semilattices and groups.  Before stating this structure theorem, we first recall the following definition.\par
Let $Y$ be a semilattice, and let $M$ be a monoid that is a disjoint union of subsemigroups $S_{\alpha}, \alpha\in Y,$ with the property that $S_{\alpha}S_{\beta}\subseteq S_{\alpha\beta}$ for all $\alpha, \beta\in Y.$  Then we say that $M$ is a {\em semilattice of semigroups} $S_{\alpha}$ $(\alpha\in Y)$ and write $M=\mathcal{S}(Y, S_{\alpha}).$  We note that $Y$ has an identity, i.e.\ it is a commutative idempotent monoid.

\begin{thm}[{\cite[Theorem 4.2.1]{Howie}}]
\label{Clifford characterisation}
Every Clifford monoid is a semilattice of groups.
\end{thm}

We note that in a Clifford monoid $M=\mathcal{S}(Y, G_{\alpha}),$ the semilattice $Y$ is isomorphic to the set $E(M)$ of idempotents of $M,$ and the $G_{\alpha}$ ($\alpha\in Y$) are the maximal subgroups of $M.$\par
The following result provides some necessary conditions for a Clifford monoid to satisfy the condition that all its finitely generated acts are completely separable.

\begin{prop}
\label{prop:Clifford,CS}
Let $M$ be a Clifford monoid, and let $M=\mathcal{S}(Y, G_{\alpha})$ be its decomposition into a semilattice of groups.  If all finitely generated $M$-acts are completely separable, then every quotient of $Y$ is completely separable and, for each $\alpha\in Y,$ the group $G_{\alpha}$ is finite.
\end{prop}

\begin{proof}
The semilattice $Y$ is isomorphic to the set of idempotents $E(M),$ which is a retract of $M$ via the homomorphism $M\to E(M), m\mapsto mm^{-1}.$  Therefore, by Proposition \ref{retract}, all finitely generated $Y$-acts are completely separable, and hence every quotient of $Y$ is completely separable by Corollary \ref{all fg acts CS}.\par
Since $\textbf{M}$ is completely separable, the monoid $M$ is completely separable by Corollary \ref{monoid separability}(4).  Therefore, each $G_{\alpha}$ is completely separable by Lemma \ref{subalgebra}, and hence each $G_{\alpha}$ is finite by Corollaries \ref{monoid separability}(4) and \ref{group}(3).
\end{proof}

The converse of Proposition \ref{prop:Clifford,CS} does not hold, as the following example demonstrates.

\begin{ex}
Let $Y=(\mathbb{N}, \max).$  For any $i\in Y,$ we can separate $i$ from $Y\!\setminus\!\{i\}$ by the congruence with classes $\{j\in Y : j<i\},$ $\{i\}$ and $\{j\in Y : i>j\}.$  Thus $Y$ is completely separable.  Morever, it can be easily shown that every quotient of $Y$ is either finite or isomorphic to $Y,$ and is hence completely separable.\par
For each $i\in\mathbb{N},$ let $G_i=\langle g_i\rangle\cong\mathbb{Z}_{2^i},$ and let $e_i$ denote the identity of $G_i$.  For each pair $i, j\in\mathbb{N}$ with $i\leq j,$ we have a natural embedding $\theta_{i, j} : G_i\to G_j$ given by $g_i\theta_{i, j}=g_j^{2^{j-i}}.$  For each $i\leq j\leq k,$ we have $\theta_{i, j}\circ\theta_{j, k}=\theta_{i, k}$ and $\theta_{i, i}$ is the identity map on $G_i$.
Let $M=\bigcup_{i\in\mathbb{N}}G_i$, and define a multiplication on $M$ by 
$$g_i^kg_j^l=(g_i^k\theta_{i, \max(i, j)})(g_j^l\theta_{j, \max(i, j)}).$$
It is straightforward to check that, under the above multiplication, $M$ is a commutative monoid and $M=\mathcal{S}(Y, G_i).$  We define a relation $\rho$ on $M$ by
$$g_i^k\,\rho\,g_j^l\iff g_i^k\theta_{i, \max(i, j)}=g_j^l\theta_{j, \max(i, j)}.$$
We shall prove that $\rho$ is a congruence on $M.$  Clearly $\rho$ is an equivalence relation.  Let $g_i^k\,\rho\,g_j^l$ and $g_p^q\in M.$  Letting $n=\max(i, j)$ and $m=\max(n, p),$ we have
\begin{align*}
(g_i^kg_p^q)\theta_{\max(i, p), m}&=\bigl((g_i^k\theta_{i, \max(i, p)})(g_p^q\theta_{p, \max(i, p)})\bigr)\theta_{\max(i, p), m}\\
&=\bigl((g_i^k\theta_{i, \max(i, p)})\theta_{\max(i, p), m}\bigr)\bigl((g_p^q\theta_{p, \max(i, p)})\theta_{\max(i, p), m}\bigr)\\
&=(g_i^k\theta_{i, m})(g_p^q\theta_{p, m})\\
&=\bigl((g_i^k\theta_{i, n})\theta_{n, m}\bigr)(g_p^q\theta_{p, m}).
\end{align*}
Similarly, we have that 
$$(g_j^lg_p^q)\theta_{\max(j, p), m}=\bigl((g_j^l\theta_{j, n})\theta_{n, m}\bigr)(g_p^q\theta_{p, m}).$$
Since $g_i^k\,\rho\,g_j^l,$ we have that $g_i^k\theta_{i, n}=g_j^l\theta_{j, n}.$  It follows that 
$$(g_i^kg_p^q)\theta_{\max(i, p), m}=(g_j^lg_p^q)\theta_{\max(j, p), m},$$
and hence $g_i^kg_p^q\,\rho\,g_j^lg_p^q,$ as required.
We claim that the cyclic $M$-act $A=\textbf{M}/\rho$ is not completely separable.  Observe that the $\rho$-class of $e_1$ is $\{e_i : i\in\mathbb{N}\}$ and that, for each $i\in\mathbb{N},$ the restriction of $\rho$ to $G_i$ is the equality relation.  Let $a=[e_1]_{\rho}.$  We shall prove that we cannot separate $a$ from $A\!\setminus\!\{a\}$ in a finite index congruence.  So, let $\sim$ be a congruence on $A$ with finite index $n.$  The elements of $G_n$ belong to distinct $\rho$-classes, so by the pigeonhole principle there exist $i, j\in\{0, \dots, 2^n-1\}$ with $i<j$ such that $[g_n^i]_{\rho}\sim[g_n^j]_{\rho}$.  Then 
$$a=[e_1]_\rho=[e_n]_{\rho}=[g_n^j]_{\rho}g_n^{2^n-j}\sim [g_n^i]_{\rho}g_n^{2^n-j}=[g_n^{2^n-j+i}]_{\rho}\neq a,$$
as required.
\end{ex}
\vspace{0.3em}
\begin{remark}~
\begin{enumerate}[leftmargin=*]
\item There exist semilattices that are not completely separable.  For example, let $Y=\{1, 0\}\cup\{x_i : i\in I\},$ where $I$ is an infinite set, and define a multiplication on $Y$ by 
$$1y=y1=y,\, 0y=y0=0,\, x_i^2=x_i,\, x_ix_j=0\: (y\in Y, i, j\in I, i\neq j).$$  
It is easily verified that $Y$ is a semilattice with identity $1.$  Now let $\sim$ be a congruence on $Y.$  Then there exist $i, j\in I$ with $i\neq j$ such that $x_i\sim x_j.$  It follows that $x_i=x_i^2\sim x_jx_i=0.$  Thus $0$ cannot be separated from $Y\!\setminus\!\{0\},$ so $Y$ is not completely separable.
\item It is possible for a semilattice $Y$ to be completely separable and have quotients that are not completely separable.  Indeed, every semilattice is the quotient of a free semilattice $\mathcal{F}(X)$ for some set $X,$ where $\mathcal{F}(X)$ is defined as the set of all finite subsets of $X$ under the operation of union.  Let $\mathcal{F}=\mathcal{F}(X)$ where $X$ is infinite.  By (1), there exist quotients of $\mathcal{F}$ that are not completely separable.  We claim that $\mathcal{F}$ is completely separable.  Indeed, let $S\in\mathcal{F}.$  Let $U$ denote the set of all subsets of $S.$  It is clear that $U$ is finite and $I=\mathcal{F}\!\setminus\!U$ is an ideal of $\mathcal{F}.$  Therefore, the Rees congruence $\rho_I$ on $\mathcal{F}$ has finite index and $[S]_{\rho_I}=\{S\},$ as required.
\end{enumerate}
\end{remark}

\begin{prob}
Which Clifford monoids $M$ satisfy the condition that all finitely generated $M$-acts are completely separable?
\end{prob}

The next main result of this section determines which Clifford monoids satisfy the condition that all their acts are completely separable.  In order to prove this result, we first provide the following proposition.

\begin{prop}
\label{semilattice of semigroups}
Let $M$ be a monoid that is a semilattice of semigroups $\mathcal{S}(Y, S_{\alpha}).$  Then the set 
$$A=\{x_{\alpha} : \alpha\in Y\}\cup\{0\}$$ is an $M$-act under the action given by 
$$0m=0,\hspace{1em}x_{\alpha}m=
\begin{cases}
x_{\alpha}&\emph{if }m\in S_{\beta}\emph{ where }\alpha\leq\beta,\\
0&\emph{otherwise.}
\end{cases}$$
Furthermore, if $Y$ is infinite then $A$ is not completely separable.
\end{prop}

\begin{proof}
We first prove that $A$ is an $M$-act.  Clearly $x_{\alpha}1=x_{\alpha}$ for all $\alpha\in Y.$  Now let $\alpha\in Y,$ $m\in S_{\beta}$ and $n\in S_{\gamma}.$  We need to show that $(x_{\alpha}m)n=x_{\alpha}(mn).$  Notice that $\alpha\leq\beta\gamma$ if and only if $\alpha\leq\beta$ and $\alpha\leq\gamma.$
Therefore, if $\alpha\leq\beta\gamma$ then
$$(x_{\alpha}m)n=x_{\alpha}n=x_{\alpha}=x_{\alpha}(mn).$$
Otherwise, we have $(x_{\alpha}m)n=0=x_{\alpha}(mn),$ as required.\par
Now assume that $Y$ is infinite.  Let $\sim$ be a finite index congruence on $A.$  Then there exist $\alpha, \beta\in Y$ with $\alpha\not\leq\beta$ such that $x_{\alpha}\sim x_{\beta}.$  For any $m\in S_{\beta},$ we have $$0=x_{\alpha}m\sim x_{\beta}m=x_{\beta}.$$
Thus $0$ cannot be separated from $A\!\setminus\!\{0\},$ so $A$ is not completely separable.
\end{proof}

\begin{thm}
Let $M$ be a Clifford monoid.  Then all $M$-acts are completely separable if and only if $M$ is finite.
\end{thm}

\begin{proof}
The reverse implication follows from Corollary \ref{finite monoid}, so we just need to prove the direct implication.  By Theorem \ref{Clifford characterisation}, $M$ is a semilattice of groups $\mathcal{S}(Y, G_{\alpha}).$  Each $G_{\alpha}$ is finite by Proposition \ref{prop:Clifford,CS}, and $Y$ is finite by Proposition \ref{semilattice of semigroups}.  We conclude that $M$ is finite.
\end{proof}

The following result will be useful for the remainder of the section.

\begin{prop}
\label{Clifford prop}
Let $M$ be a Clifford monoid, let $A$ be an $M$-act, and let $a, b\in A.$  If $(a, b)\notin\mathcal{R}_A,$ then there exists a finite index congruence on $A$ that separates $a$ and $b.$
\end{prop}

\begin{proof}
Let $\leq$ denote the preorder $\leq_A$.  We may assume without loss of generality that $a\not\leq b.$  Let $B=\{x\in A : a\not\leq x\}.$  Then $B$ is a subact of $A.$  Indeed, if $B$ were not a subact, then there would exist $x\in B$ and $m\in M$ such that $a\leq xm,$ but then would exist $n\in M$ such that $a=(xm)n=x(mn),$ contradicting that $a\not\leq x.$\par
We now define a relation $\sim$ on $A$ by 
$$x\sim y\iff x, y\in A\!\setminus\!B\text{ or }x, y\in B.$$
We claim that $\sim$ is a congruence on $A.$  Clearly $\sim$ is an equivalence relation.  Now let $x\sim y$ and $m\in M.$  If $x, y\in B,$ then $xm, ym\in B$ since $B$ is a subact of $A.$  Assume then that $x, y\in A\!\setminus\!B.$  We just need to prove that $xm\in A\!\setminus\!B$ if and only if $ym\in A\!\setminus\!B$; that is, $a\leq xm$ if and only if $a\leq ym.$  Suppose that $a\leq xm.$  Then there exist $p, q\in M$ such that $a=yp=(xm)q.$  Since $M$ is regular, there exists $u\in M$ such that $mq=(mq)u(mq).$  Letting $e$ denote the idempotent $mqu,$ and using the fact that idempotents of $M$ are central, we deduce that
$$(ym)(qup)=y(mqup)=yep=ype=x(mq)e=xe(mq)=x(mq)=a,$$
so $a\leq ym.$  By symmetry, if $a\leq ym$ then $a\leq xm.$  Thus $\sim$ is a congruence.  Clearly $\sim$ has index 2.  Finally, we have $a\in A\!\setminus\!B$ and $b\in B,$ so $\sim$ separates $a$ and $b,$ as required.
\end{proof}

We can now readily prove that all acts over a Clifford monoid are strongly subact separable.

\begin{thm}
\label{Clifford SSS}
Let $M$ be a Clifford monoid.  Then all $M$-acts are strongly subact separable.
\end{thm}

\begin{proof}
Let $A$ be an $M$-act with a zero $0,$ and consider an element $a\in A\!\setminus\!\{0\}.$  By Proposition \ref{all acts SSS}, it suffices to prove that we can separate $a$ and $0.$
Clearly $(a, 0)\notin\mathcal{R}_A$ since $\langle 0\rangle=\{0\}.$  Therefore, by Proposition \ref{Clifford prop} there exists a finite index congruence on $A$ that separates $a$ and $0,$ as required.
\end{proof}

\begin{corollary}
Let $M$ be a commutative idempotent monoid.  Then all $M$-acts are residually finite and strongly subact separable.
\end{corollary}

\begin{proof}
Since $M$ is a Clifford monoid, all $M$-acts are strongly subact separable by Theorem \ref{Clifford SSS}.  It then follows from Lemma \ref{Green} and Remark \ref{Green remark} that all $M$-acts are residually finite.
\end{proof}

\begin{remark}
Kozhukhov proved that a non-trivial monoid $M$ has the property that every $M$-act is a subdirect product of two element acts if and only if $M$ is a commutative idempotent monoid \cite[Theorem 2]{Kozhukhov1}.  It can be easily deduced from this result that every act over a commutative idempotent monoid is residually finite.
\end{remark}

In the remainder of this section we consider Clifford monoids whose (finitely generated) acts are all residually finite.  The following result states that their maximal subgroups are strongly subgroup separable.

\begin{prop}
\label{max subgroups}
Let $M$ be a Clifford monoid.  If all finitely generated $M$-acts are residually finite, then every maximal subgroup of $M$ is strongly subgroup separable.
\end{prop}

\begin{proof}
Let $M=\mathcal{S}(Y, G_{\alpha})$ be the decomposition of $M$ into a semilattice of groups.  Let $\alpha\in Y.$  We need to show that $G_{\alpha}$ is strongly subgroup separable.  Let $N=\bigcup_{\beta\in Y^{\prime}}G_{\beta},$ where $Y^{\prime}=\{\beta\in Y : \alpha\leq\beta\}.$  Then $N$ is a submonoid of $M$ and $M\!\setminus\! N$ is an ideal.  The group $G_{\alpha}$ is easily seen to be a retract of $N$ via the homomorphism $N\to G_{\alpha}, n\mapsto n1_{\alpha}.$  Therefore, by Propositions \ref{submonoid} and \ref{retract}, all finitely generated $G_{\alpha}$-acts are residually finite.  It now follows from Theorem \ref{all G-acts RF} that $G_{\alpha}$ is strongly subgroup separable.
\end{proof}

The following questions remain open.

\begin{prob}
\label{Clifford op}
Let $M$ be a Clifford monoid.  
\begin{enumerate}[leftmargin=*]
\item If every maximal subgroup of $M$ is strongly subgroup separable, are all $M$-acts residually finite?
\item If every maximal subgroup of $M$ is strongly subgroup separable, are all finitely generated $M$-acts residually finite?
\item Are all $M$-acts residually finite if and only if all finitely generated $M$-acts are residually finite?
\end{enumerate}
\end{prob}

\begin{remark}
\label{op remark}
It is obvious that a positive solution to (1) of Open Problem \ref{Clifford op} would imply a positive solution to (2).  Note that, given Proposition \ref{max subgroups}, a positive solution to (1) would also imply a positive solution to (3).
\end{remark}

The following result provides a partial solution to Open Problem \ref{Clifford op}.

\begin{thm}
Let $M$ be a Clifford monoid, and let $M=\mathcal{S}(Y, G_{\alpha})$ be its decomposition into a semilattice of groups.  If every subsemilattice of $Y$ has a least element, then the following are equivalent:
\begin{enumerate}
\item all $M$-acts are residually finite;
\item all finitely generated $M$-acts are residually finite;
\item for each $\alpha\in Y,$ the group $G_{\alpha}$ is strongly subgroup separable.
\end{enumerate}
\end{thm}

\begin{proof}
By Remark \ref{op remark}, we only need to prove that (3) implies (1).
So, let $A$ be an $M$-act, and let $a, b\in A$ with $a\neq b.$  If $(a, b)\notin\mathcal{R}_A,$ then by Proposition \ref{Clifford prop} there exists a finite index congruence on $A$ that separates $a$ and $b.$
Assume then that $(a, b)\in\mathcal{R}_A.$ 
Let $$Y^{\prime}=\{\alpha\in Y : a=ag\text{ for some }g\in G_{\alpha}\}.$$
If $\alpha, \beta\in Y^{\prime}$ then $a=ag=ah$ for some $g\in G_{\alpha}, h\in G_{\beta},$ and hence $a=agh,$ so $\alpha\beta\in Y^{\prime}.$  Thus $Y^{\prime}$ is a subsemilattice of $Y.$  By the assumption, $Y^{\prime}$ has a least element, say $\lambda.$  Note that if $\alpha\in Y^{\prime},$ then since $a=ag$ for some $g\in G_{\lambda},$ we have that $a1_{\alpha}=ag1_{\alpha}=ag=a.$\par
Let $R$ denote the $\mathcal{R}_A$-class of $a.$  Consider $x\geq a.$  Then $a=xg$ for some $g\in G_{\alpha}, \alpha\in Y,$ and thus $a=xg1_{\alpha}=a1_{\alpha},$ so $\alpha\in Y^{\prime}.$  It follows that 
$$a=a1_{\lambda}=(xg)1_{\lambda}=(x1_{\lambda})g\,\text{ and }\,x1_{\lambda}=x1_{\alpha}1_{\lambda}=xgg^{-1}1_{\lambda}=a(g^{-1}1_{\lambda}),$$
so $x1_{\lambda}\in R.$  Suppose further that $x\in R.$  Then $x=ah$ for some $h\in G_{\beta}, \beta\in Y.$  It follows that $$x1_{\lambda}=ah1_{\lambda}=a1_{\lambda}h=ah=x.$$
Now let $k\in G_{\lambda}.$  Then $a(hk)=xk$ and $(xk)(k^{-1}g)=x1_{\lambda}g=xg=a,$ so $xk\in R.$  Therefore, the set $R$ is a $G_{\lambda}$-act.  By the assumption that $G_{\lambda}$ is strongly subgroup separable, along with Theorem \ref{all G-acts RF}, there exist a finite $G_{\lambda}$-act $C$ and a $G_{\lambda}$-homomorphism $\theta : R\to C$ such that $a\theta\neq b\theta.$
Let $B=\{x\in A : a\not\leq x\}.$  We define a relation $\sim$ on $A$ by 
$$x\sim y\iff
\begin{cases}
x, y\in A\!\setminus\!B\text{ and }(x1_{\lambda})\theta=(y1_{\lambda})\theta,\text{\,or}\\
x, y\in B.
\end{cases}$$
We claim that $\sim$ is a congruence on $A.$  Clearly $\sim$ is an equivalence relation.  Now let $x\sim y$ and $m\in M.$  If $x, y\in B,$ then $xm, ym\in B$ since $B$ is a subact of $A.$  Assume then that $x, y\in A\!\setminus\!B$ and $(x1_{\lambda})\theta=(y1_{\lambda})\theta.$  As in the proof of Proposition \ref{Clifford prop}, we have that $xm\in A\!\setminus\!B$ if and only if $ym\in A\!\setminus\!B.$  Assume then that $xm\in A\!\setminus\!B.$  Then there exists $n\in M$ such that $(xm)n=a.$  We have that $m\in G_{\alpha}$ and $h\in G_{\beta}$ for some $\alpha, \beta\in Y.$  Since $a=a1_{\alpha\beta},$ we have $\alpha\beta\in Y^{\prime},$ and hence $$\alpha\lambda=\alpha\alpha\beta\lambda=\alpha\beta\lambda=\lambda.$$  Thus $m1_{\lambda}\in G_{\lambda}.$  It follows that
\begin{align*}
\bigl((xm)1_{\lambda}\bigr)\theta&=\bigl((x1_{\lambda})(m1_{\lambda})\bigr)\theta=\bigl((x1_{\lambda})\theta\bigr)(m1_{\lambda})=\bigl((y1_{\lambda})\theta\bigr)(m1_{\lambda})=\bigl((y1_{\lambda})(m1_{\lambda})\bigr)\theta\\&=\bigl((ym)1_{\lambda}\bigr)\theta,
\end{align*}
so $xm\sim ym.$  Thus $\sim$ is a congruence.  Clearly $\sim$ has finite index since $C$ is finite.  Finally, we have that $a, b\in A\!\setminus\!B$ and $$(a1_{\lambda})\theta=a\theta\neq b\theta=(b1_{\lambda})\theta,$$
so $\sim$ separates $a$ and $b.$  This completes the proof.
\end{proof}

\begin{corollary}
Let $M$ be a Clifford monoid with finitely many idempotents.  Then the following are equivalent:
\begin{enumerate}
\item all $M$-acts are residually finite;
\item all finitely generated $M$-acts are residually finite;
\item every maximal subgroup of $M$ is strongly subgroup separable.
\end{enumerate}
\end{corollary}

\section*{Acknowledgements}
This work was developed within the activities of CEMAT (Centro de Matem{\'a}tica Computacional e Estoc{\'a}stica) and Departamento de Matem{\'a}tica da Faculdade de Ci{\^e}ncias da Universidade de Lisboa, under the projects UIDB/04621/2020 and UIDP/04621/2020, financed by Funda{\c c}{\~a}o para a Ci{\^e}ncia e a Tecnologia.\par
The author would like to thank the referee for their helpful comments and suggestions.

\end{document}